\documentclass[twoside]{aiml22}

\usepackage{aiml22macro}

\usepackage{graphicx}
\usepackage{amsmath}
\usepackage{amssymb}

\usepackage{tikz}
\usepackage{bussproofs}
\usepackage{amssymb}
\usepackage{multirow}
\usepackage{stmaryrd}
\usepackage{mathrsfs}
\usepackage{lipsum} 
\usepackage{enumitem}
\usepackage{scrextend}

\newcommand{\mf}{\mathsf}

\newcommand{\ufor}{\Vd^{\forall}}
\newcommand{\efor}{\Vd^{\exists}}
\newcommand{\leftrule}[1]{$\mathsf{#1_{L}}$}
\newcommand{\rightrule}[1]{$\mathsf{#1_{R}}$}

\newcommand{\init}{$\mf{init}$}
\newcommand{\lbot}{\leftrule{\bot}}

\newcommand{\lland}{\leftrule{\land}}
\newcommand{\rland}{\rightrule{\land}}
\newcommand{\llor}{\leftrule{\lor}}
\newcommand{\rlor}{\rightrule{\lor}}

\newcommand{\limp}{\leftrule{\imp}}
\newcommand{\rimp}{\rightrule{\imp}}

\newcommand{\diam}{\Diamond}
\newcommand{\boxm}[1]{\Box^-#1}

\newcommand{\G}{\Gamma}
\newcommand{\D}{\Delta}

\newcommand{\N}{\mathcal N}

\newcommand{\W}{\mathcal W}

\newcommand{\V}{\mathcal V}

\newcommand{\R}{\mathcal R}

\newcommand{\CNM}{CNM}

\newcommand{\M}{\mathcal M}

\newcommand{\lan}{\mathcal L}

\newcommand{\atm}{Atm}

\newcommand{\CC}{\mathscr{C}}
\newcommand{\U}{\mathscr{U}}
\newcommand{\aU}{\alpha_{\mathscr{U}}}
\newcommand{\aUp}{\alpha_{\mathscr{U'}}}
\newcommand{\canonic}[1]{#1}
\newcommand{\Mc}{\canonic{\M}}
\newcommand{\Vc}{\canonic{\V}}
\newcommand{\Nc}{\canonic{\N}}
\newcommand{\Wc}{\canonic{\W}}
\newcommand{\lessc}{\canonic{\less}}
\newcommand{\morec}{\canonic{\more}}

\newcommand{\Sc}{\Sigma'_{C}}
\newcommand{\Sd}{\Sigma'_{D}}
\newcommand{\Scd}{\Sigma'_{CD}}
\newcommand{\Stopd}{\Sigma'_{\top D}}

\newcommand{\Sb}{\Sigma'_{B}}
\newcommand{\Sab}{\Sigma'_{AB}}
\newcommand{\Ud}{\U'_{D}}
\newcommand{\Ue}{\U'_{E}}
\newcommand{\Ua}{\U_{A}}
\newcommand{\Ub}{\U_{B}}

\newcommand{\pow}{\mathcal P}

\newcommand{\E}{\logicnamestyle{E}}

\newcommand{\EM}{\logicnamestyle{M}}

\newcommand{\EMN}{\logicnamestyle{MN}}
\newcommand{\EMC}{\logicnamestyle{MC}}

\newcommand{\EMP}{\logicnamestyle{MP}}
\newcommand{\EMNP}{\logicnamestyle{MNP}}
\newcommand{\EMCP}{\logicnamestyle{MCP}}

\newcommand{\EMD}{\logicnamestyle{MD}}
\newcommand{\EMND}{\logicnamestyle{MND}}
\newcommand{\EMCD}{\logicnamestyle{MCD}}

\newcommand{\EMT}{\logicnamestyle{MT}}
\newcommand{\EMNT}{\logicnamestyle{MNT}}
\newcommand{\EMCT}{\logicnamestyle{MCT}}

\newcommand{\K}{\logicnamestyle{K}}
\newcommand{\KT}{\logicnamestyle{KT}}
\newcommand{\KD}{\logicnamestyle{KD}}
\newcommand{\KP}{\logicnamestyle{KP}}

\newcommand{\IK}{\logicnamestyle{IK}}

\newcommand{\WK}{\logicnamestyle{WK}}
\newcommand{\WKD}{\logicnamestyle{WKD}}
\newcommand{\WKT}{\logicnamestyle{WKT}}
\newcommand{\WM}{\logicnamestyle{WM}}
\newcommand{\WMN}{\logicnamestyle{WMN}}
\newcommand{\WMC}{\logicnamestyle{WMC}}
\newcommand{\WMD}{\logicnamestyle{WMD}}
\newcommand{\WMND}{\logicnamestyle{WMND}}
\newcommand{\WMCD}{\logicnamestyle{WMCD}}
\newcommand{\WMT}{\logicnamestyle{WMT}}
\newcommand{\WMNT}{\logicnamestyle{WMNT}}
\newcommand{\WMCT}{\logicnamestyle{WMCT}}
\newcommand{\WMP}{\logicnamestyle{WMP}}
\newcommand{\WMNP}{\logicnamestyle{WMNP}}

\newcommand{\seqWK}{\Gone{WK}}
\newcommand{\seqWKD}{\Gone{WKD}}
\newcommand{\seqWKT}{\Gone{WKT}}
\newcommand{\seqWM}{\Gone{WM}}
\newcommand{\seqWMN}{\Gone{WMN}}
\newcommand{\seqWMC}{\Gone{WMC}}
\newcommand{\seqWMD}{\Gone{WMD}}
\newcommand{\seqWMND}{\Gone{WMND}}
\newcommand{\seqWMCD}{\Gone{WMCD}}
\newcommand{\seqWMT}{\Gone{WMT}}
\newcommand{\seqWMNT}{\Gone{WMNT}}
\newcommand{\seqWMCT}{\Gone{WMCT}}
\newcommand{\seqWMP}{\Gone{WMP}}
\newcommand{\seqWMNP}{\Gone{WMNP}}
\newcommand{\seqWL}{\Gone{WL}}
\newcommand{\seqWstar}{\Gone{W^*}}

\newcommand{\seqK}{\Gone{K}}
\newcommand{\seqKD}{\Gone{KD}}
\newcommand{\seqKT}{\Gone{KT}}
\newcommand{\seqEM}{\Gone{M}}
\newcommand{\seqEMN}{\Gone{MN}}
\newcommand{\seqEMC}{\Gone{MC}}
\newcommand{\seqEMD}{\Gone{MD}}
\newcommand{\seqEMND}{\Gone{MND}}
\newcommand{\seqEMCD}{\Gone{MCD}}
\newcommand{\seqEMT}{\Gone{MT}}
\newcommand{\seqEMNT}{\Gone{MNT}}
\newcommand{\seqEMCT}{\Gone{MCT}}
\newcommand{\seqEMP}{\Gone{MP}}
\newcommand{\seqEMNP}{\Gone{MNP}}

\newcommand{\Sfive}{\logicnamestyle{S5}}

\newcommand{\CPL}{\logicnamestyle{CPL}}
\newcommand{\IPL}{\logicnamestyle{IPL}}

\newcommand{\hilbertaxiomstyle}[1]{${#1}$}

\newcommand{\axT}{\hilbertaxiomstyle{T}}
\newcommand{\axD}{\hilbertaxiomstyle{D}}

\newcommand{\axfour}{\hilbertaxiomstyle{4}}
\newcommand{\axfive}{\hilbertaxiomstyle{5}}
\newcommand{\axB}{\hilbertaxiomstyle{B}}

\newcommand{\ax}{\AxiomC}
\newcommand{\uinf}{\UnaryInfC}
\newcommand{\binf}{\BinaryInfC}

\newcommand{\llab}{\LeftLabel}
\newcommand{\rlab}{\RightLabel}
\newcommand{\disp}{\DisplayProof}

\newcommand{\semcond}[1]{{#1}}

\newcommand{\cN}{\semcond N}

\newcommand{\cC}{\semcond C}
\newcommand{\cT}{\semcond T}
\newcommand{\cP}{\semcond P}
\newcommand{\cD}{\semcond D}

\newcommand{\cX}{\semcond X}

\newcommand{\seq}{\Rightarrow}

\newcommand{\tto}{\imp\coimp}
\newcommand{\AND}{\bigwedge}
\newcommand{\OR}{\bigvee}
\newcommand{\vd}{\vdash}
\newcommand{\Vd}{\Vdash}
\newcommand{\eg}{e.g.}
\newcommand{\ie}{i.e.}
\newcommand{\ih}{i.h.}
\newcommand{\Gone}[1]{\mathsf{S.{#1}}}

\newcommand{\wij}{Wijesekera}

\newcommand{\intu}{intuitionistic}
\newcommand{\const}{constructive}
\newcommand{\neigh}{neighbourhood}
\newcommand{\seg}{segment}

\newenvironment{proof-sketch}{\noindent{\em Sketch of Proof.}\hspace*{0.5em}}{\qed\medskip}%{\qed\medskip}%{\qed\bigskip}

\newcommand{\cut}{$\mf{cut}$}

\newcommand{\axCdiam}{\hilbertaxiomstyle{C_\Diamond}}
\newcommand{\axNdiam}{\hilbertaxiomstyle{N_\Diamond}}

\newcommand{\axCbox}{\hilbertaxiomstyle{C_\Box}}
\newcommand{\axNbox}{\hilbertaxiomstyle{N_\Box}}

\newcommand{\axKdiam}{\hilbertaxiomstyle{K_\Diamond}}
\newcommand{\axKbox}{\hilbertaxiomstyle{K_\Box}}
\newcommand{\nec}{\hilbertaxiomstyle{nec}}
\newcommand{\monbox}{\hilbertaxiomstyle{mon_\Box}}
\newcommand{\mondiam}{\hilbertaxiomstyle{mon_\diam}}

\newcommand{\axTbox}{\hilbertaxiomstyle{T_\Box}}
\newcommand{\axTdiam}{\hilbertaxiomstyle{T_\diam}}
\newcommand{\axPbox}{\hilbertaxiomstyle{P_\Box}}
\newcommand{\axPdiam}{\hilbertaxiomstyle{P_\diam}}

\newcommand{\axdual}{\hilbertaxiomstyle{dual}}
\newcommand{\axdualand}{\hilbertaxiomstyle{dual_\land}}
\newcommand{\Rdualand}{\hilbertaxiomstyle{Rdual_\land}}
\newcommand{\axdualor}{\hilbertaxiomstyle{dual_\lor}}
\newcommand{\Rdualor}{\hilbertaxiomstyle{Rdual_\lor}}

\newcommand{\rulestyle}[1]{$\mathsf{#1}$}

\newcommand{\ruledualandaux}[1]{\rulestyle{\land\textup{-}dual_{#1}}}
\newcommand{\ruledualoraux}[1]{\rulestyle{\lor\textup{-}dual_{#1}}}
\newcommand{\ruledualandauxtwo}[2]{\rulestyle{\land\textup{-}dual_{#1}^{#2}}}
\newcommand{\ruledualandM}{\ruledualandaux{M}}
\newcommand{\ruledualorM}{\ruledualoraux{M}}
\newcommand{\ruledualandC}{\ruledualandaux{C}}
\newcommand{\ruledualorC}{\ruledualoraux{C}}

\newcommand{\ruledualorR}{\ruledualoraux{R}}
\newcommand{\ruleKbox}{\rulestyle{K_\Box}}
\newcommand{\ruleKdiam}{\rulestyle{K_\diam}}
\newcommand{\ruleNbox}{\rulestyle{N_\Box}}
\newcommand{\ruleNdiam}{\rulestyle{N_\diam}}
\newcommand{\ruleMbox}{\rulestyle{M_\Box}}
\newcommand{\ruleMdiam}{\rulestyle{M_\diam}}
\newcommand{\ruleCbox}{\rulestyle{C_\Box}}
\newcommand{\ruleCdiam}{\rulestyle{C_\diam}}
\newcommand{\ruleTbox}{\rulestyle{T_\Box}}
\newcommand{\ruleTdiam}{\rulestyle{T_\diam}}

\newcommand{\ruleD}{\rulestyle{D}}
\newcommand{\rulePbox}{\rulestyle{P_\Box}}
\newcommand{\rulePdiam}{\rulestyle{P_\diam}}
\newcommand{\ruleCD}{\rulestyle{CD}}

\newcommand{\ruleDbox}{\rulestyle{D_\Box}}
\newcommand{\ruleDdiam}{\rulestyle{D_\diam}}
\newcommand{\ruleiKbox}{\rulestyle{K_\Box^i}}
\newcommand{\ruleiKdiam}{\rulestyle{K_\diam^i}}
\newcommand{\ruleiNbox}{\rulestyle{N_\Box^i}}
\newcommand{\ruleiNdiam}{\rulestyle{N_\diam^i}}
\newcommand{\ruleiMbox}{\rulestyle{M_\Box^i}}
\newcommand{\ruleiMdiam}{\rulestyle{M_\diam^i}}
\newcommand{\ruleiCbox}{\rulestyle{C_\Box^i}}
\newcommand{\ruleiCdiam}{\rulestyle{C_\diam^i}}
\newcommand{\ruleiTbox}{\rulestyle{T_\Box^i}}
\newcommand{\ruleiTdiam}{\rulestyle{T_\diam^i}}

\newcommand{\ruleiD}{\rulestyle{D^i}}

\newcommand{\ruleiDbox}{\rulestyle{D_\Box^i}}

\newcommand{\ruleiPbox}{\rulestyle{P_\Box^i}}
\newcommand{\ruleiPdiam}{\rulestyle{P_\diam^i}}
\newcommand{\ruleiCD}{\rulestyle{CD^i}}
\newcommand{\ruleiCDbox}{\rulestyle{CD_\Box^i}}

\newcommand{\less}{\leq}
\newcommand{\more}{\geq}

\newcommand{\imp}{\supset}
\newcommand{\coimp}{\subset}

\newcommand{\var}{var}

\newcommand{\vleadsto}{\mathbin{\rotatebox[origin=c]{270}{$\leadsto$}}}

\newcommand{\logicnamestyle}[1]{\mathsf{#1}}
\newcommand{\CK}{\logicnamestyle{CK}}

\newcommand{\logic}{\logicnamestyle{L}}
\newcommand{\seqlogic}{\logicnamestyle{S.L}}
\newcommand{\WL}{\Wlogic}
\newcommand{\WLL}{\logicnamestyle{WL}}

\newcommand{\WPstar}{\logicnamestyle{WP^*}}
\newcommand{\WNstar}{\logicnamestyle{WN^*}}
\newcommand{\WCstar}{\logicnamestyle{WC^*}}
\newcommand{\WDstar}{\logicnamestyle{WD^*}}
\newcommand{\WTstar}{\logicnamestyle{WT^*}}
\newcommand{\Wlogic}{\logicnamestyle{W^*}}
\newcommand{\seqWlogic}{\logicnamestyle{S.WL}}
\newcommand{\varax}{\hilbertaxiomstyle{X}}
\newcommand{\varaxbox}{\hilbertaxiomstyle{X_\Box}}

\newcommand{\ruleidualandM}{\ruledualandauxtwo{M}{i}}
\newcommand{\ruleidualandC}{\ruledualandauxtwo{C}{i}}
\newcommand{\ruleidualandK}{\ruledualandauxtwo{K}{i}}

\newcommand{\iinit}{$\mf{init^i}$}
\newcommand{\ilbot}{\leftrule{\bot^i}}
\newcommand{\illand}{\leftrule{\land^i}}
\newcommand{\irland}{\rightrule{\land^i}}
\newcommand{\illor}{\leftrule{\lor^i}}
\newcommand{\irlor}{\rightrule{\lor^i}}
\newcommand{\ilimp}{\leftrule{\imp^i}}
\newcommand{\irimp}{\rightrule{\imp^i}}
\newcommand{\ilwk}{$\mathsf{Lwk}$}
\newcommand{\irwk}{$\mathsf{Rwk}$}
\newcommand{\ilctr}{$\mathsf{ctr}$}

\newcommand\blfootnote[1]{%
  \begingroup
  \renewcommand\thefootnote{}\footnote{#1}%
  \addtocounter{footnote}{-1}%
  \endgroup
}

\def\lastname{Dalmonte}

\begin{document}

\begin{frontmatter}
%\title{Constructive modal logics: from proof theory to semantics}
%\title{A family of constructive modal logics: from proof theory to semantics}
%\title{Wijesekera-style constructive modal logics: from proof theory to semantics}
\title{Wijesekera-style constructive modal logics}
  \author{Tiziano Dalmonte} %\footnote{tiziano.dalmonte@gmail.com} %\quad You can put your email address or grant acknowledgement as a footnote, if you wish.}
%\address{Center for Logic, Language, and Cognition, Department of Philosophy and Education, University of Turin, Turin, Italy}
\address{Free University of Bozen-Bolzano, Bolzano, Italy}
%  \author{Second Author}
%  \address{Affiliation \\ Address \\ Address}

  \begin{abstract}
  %Abstract of approximately 100-200 words.
We define a family of propositional constructive modal logics
corresponding each to a different classical modal system.
The logics are 
defined %constructed
in the style of Wijesekera's constructive modal logic~\cite{wij}, and are
both proof-theoretically and semantically motivated.
On the one hand, 
they correspond to the single-succedent restriction %s 
of standard sequent calculi for classical modal logics.
%they are obtained by restricting standard sequent calculi for classical modal logics to single-succedent sequents.
On the other hand, they are obtained 
by incorporating the hereditariness of intuitionistic Kripke models into the classical satisfaction clauses for modal formulas.
%by considering a natural generalisation of 
%the classical satisfaction clauses over intuitionistic Kripke models,
%according to which the classical satisfaction clauses for modal formulas are required to hold 
%for all successors of a given world.
We show that, for the considered classical logics,
the proof-theoretical and the semantical approach %define
return
%(exactly)
the same constructive systems.
  \end{abstract}

 \end{frontmatter}

%\begin{center}
%\begin{tabular}{lllllll}
%($\Vd\Box$) &\ \ & $\M, w\Vd\Box A$ &\quad & iff &\quad & for all $v\more w$, for all $u$, if $v\R u$ then $\M, u\Vd A$. \\
%($\Vd\diam$) &\ \ & $\M, w\Vd\diam A$ &\quad & iff &\quad & for all $v\more w$, there is $u$ such that $v\R u$ and $\M, u\Vd A$. \\
%\end{tabular}
%\end{center}

%\begin{definition}
%A \emph{minimal model} is a tuple $\M = \langle \W, \less, \R, \FW, \V\rangle$, 
%where 
%%$\W$ is a non-empty set of items, called worlds, 
%$\W$ is a non-empty set of worlds, 
%%$\less$ is a reflexive and transitive relation on $\W$, 
%$\less$ is a preorder on $\W$, 
%$\R$ is a binary relation on $\W$, 
%$\FW\subseteq\W$ is a $\less$-upward closed set of worlds
%(\ie, if $w\in\FW$ and $w\less v$, then $v\in\FW$), and
%$\V : \atm \longrightarrow \pow(\W)$ is a hereditary valuation function
%(\ie, if $w\in\V(p)$ and $w\less v$, then $v\in\V(p)$).
%\end{definition}

\section{Introduction}

Constructive%
\blfootnote{Preprint accepted for publication at AiML 2022. This version corrects typos in Defs.~\ref{def:neigh model} and \ref{def:semantics}
and on p.~\pageref{correction1}.}
 or intuitionistic modal logics are extensions of \intu\ logic with modalities $\Box$ and $\diam$.
The  motivations for the study of modalities with an \intu\ basis are manifold,
%but they can be summarised into two orders of reasons.
but they can be schematically classified into two kinds. % of reasons.
On the one hand, from a theoretical perspective,
it comes natural to combine intuitionistic and modal logic \cite{Simpson:1994},
considering in particular that both of them can be semantically arranged in terms of %by means of 
possible world models.
%At the same time,
In addition, the rejection of classical equivalences can allow for a finer analysis of the modalities.
On the other hand, intuitionistic or constructive modal logics can be motivated by specific applications in computer science,
such as type-theoretic interpretations, %\cite{x}, 
verification, %\cite{x} 
and knowledge representation. % \cite{x}.

%The motivation for the study of modalities with %on
%an \intu\ basis is twofold.
%From a theoretical perspective,
%it comes natural to combine intuitionistic and modal logic \cite{x},
%considering in particular that both of them can be semantically arranged in terms of %by means of 
%possible world models.
%At the same time, the rejection of classical equivalences can allow for a finer analysis of the modalities.
%On the other hand, intuitionistic or constructive modal logics can be motivated by specific applications in computer science,
%such as type-theoretic interpretations \cite{x}, verification \cite{x} and knowledge representation \cite{x}.

%The study of modalities in an intuitionistic or constructive setting has a long tradition.

%\bigskip

A %distinguishing 
peculiar feature
%aspect
of intuitionistic modal logics is that, 
%A peculiar feature of \intu\ modalities is that, similarly to the propositional connectives, they are not interdefinable.
%A peculiar feature of modal logics with an \intu\ basis is that, 
similarly to the %other
\intu\  connectives, $\Box$ and $\diam$ are not interdefinable.
%ON THE ONE HAND
This allows for the definition of %modal 
systems 
in which $\Box$ and $\diam$ satisfy %different 
distinct principles.
%ON THE OTHER HAND
At the same time, it makes possible to define different \intu\ or \const\ counterparts of the same classical logic,
%This is 
%as it is 
%clearly 
%manifested by the three systems $\CK$\cite{ref}, WIJ \cite{ref}, and $\IK$ \cite{ref} 
%which, despite being non-equivalent, can be equally seen as counterparts of classical $\K$.
as it is testified % witnessed
 by the several \intu\ versions of classical $\K$ which have been proposed in the literature %\cite{x,y,z}
(see \cite{Simpson:1994} for a survey).

%The intuitionistic modal logics studied in the literature are 
Intuitionistic modal logics have been formulated as 
%either 
monomodal 
%(with either $\Box$ or $\diam$)
(with only $\Box$ or only $\diam$)
or bimodal (with both $\Box$ and $\diam$)
systems.
%Both monomodal (with only $\Box$ or only $\diam$),
%and bimodal (with both $\Box$ and $\diam$)
%have been studied.
Considering logics including both modalities, % $\Box$ and $\diam$, 
two \intu\ versions of $\K$ have been mostly considered:
so-called Intuitionistic K ($\IK$) and Constructive K ($\CK$).
The first system was introduced by Fischer Servi~\cite{FischerServi:1980}, Ploktin and Stirling~\cite{Plotkin}, and Ewald~\cite{Ewald},
and can be 
%seen 
defined 
as the set of formulas %such that their 
whose standard translation is derivable in first-order \intu\ logic~\cite{Simpson:1994,Wolter:1999}.
The second system,
which is weaker that $\IK$, was introduced by Bellin, de Paiva and Ritter~\cite{Bellin},
and was motivated by type-theoretic interpretations of the modalities and categorical semantics,
but also by contextual reasoning \cite{Mendler1,Mendler:2014}.  
Both the
%The 
semantics~\cite{FischerServi:1980,Plotkin,Ewald,Simpson:1994,Balbiani,Mendler1,Alechina,dal:JPL,Mendler:tableaux,Acclavio}
and the proof theory~\cite{Amati,Strassburger:2013,Galmiche:2018,Kuznets,Marin:2017,Marin:2021,Lyon,Bellin,Arisaka,Mendler2,dal:tableaux}
of 
%both
$\IK$ and $\CK$ have been extensively investigated,
in particular significant %attention 
consideration
has been devoted to their extensions % the extensions of these logics 
with standard modal axioms \axD, \axT, \axB, \axfour, and \axfive,
so that entire
families of \intu\ and \const\ modal logics are now available in the literature.

%The tw systems have been studied both sema

An additional bimodal constructive version of $\K$ was proposed by \wij~\cite{wij}.
\wij's logic aimed at representing reasoning with partial information about the states of 
concurrent transition systems, and was 
%initially
introduced as a modal extension of first-order intuitionistic logic.
If we restrict our attention to its propositional fragment,
\wij's logic 
%turns out to be 
is 
intermediate between $\CK$ and $\IK$.
%An additional bimodal constructive version of $\K$ was proposed by \wij~\cite{wij}
%to represent reasoning with partial information about the states of 
%%a concurrent transition system.
%concurrent transition systems.
%\wij's logic was initially introduced as a modal extension of first-order intuitionistic logic,
%however 
%%if we restrcit our attention analysis
%%to 
%%focusing only on
%restricting to
%its propositional fragment,
%this system
%turns out to be intermediate between $\CK$ and $\IK$.
%In this paper we will focus on the propositional fragment of \wij's logic %, that we call $\WK$,
%(we call it $\WK$),
%%and 
%which is defined by
In particular, %the propositional fragment of 
\wij's logic
%this system
(we call it $\WK$)
can be defined by
extending (any axiomatisation of) \intu\ propositional logic  ($\IPL$)
with the following modal axioms and rules:%
\footnote{\wij~\cite{wij} includes also
 the axiom $\Box A \land \diam(A \imp B) \imp \diam B$ which is derivable from the others.}
\begin{center}
\begin{tabular}{llllllll}
\multirow{2}{*}{\ax{$A$}
\llab{\nec}
\uinf{$\Box A$}
\disp} &&
\axKbox &  $\Box(A \imp B) \imp (\Box A \imp \Box B)$ && \axNdiam & $\neg\diam\bot$ \\
 &&
\axKdiam &  $\Box(A \imp B) \imp (\diam A \imp \diam B)$ \\
\end{tabular}
\end{center}
%
%($\CK$ is obtained by dropping $\neg\diam\bot$, 
%$\IK$ is obtained by adding to $\WK$ the axioms
%$\diam(A \lor B) \imp \diam A \lor \diam B$ and
%$(\diam A \imp \Box B) \imp \Box(A \imp B)$).
Then $\CK$ can be obtained by dropping $\neg\diam\bot$, % from $\WK$,
whereas $\IK$ is obtained by extending $\WK$ with the axioms
$\diam(A \lor B) \imp \diam A \lor \diam B$ and
$(\diam A \imp \Box B) \imp \Box(A \imp B)$.
%%%This logic 
%%Besides its intended interpretation,
%%$\WK$ also reveals some %interesting 
%%remarkable semantical and proof-theoretical aspects
%%%features
%%that will be object of our attention.
The interest of $\WK$ is not limited to its intended interpretation:
this logic also 
%displays
exhibits
an elegant relation with classical $\K$, 
both from a semantical and from a proof-theoretical perspective. %point of view.
We now illustrate this relation.
%reveals some %interesting 
%remarkable semantical and proof-theoretical aspects
%that will be object of our attention.
%One of the reason of interest of \wij's logic
%relies on the way in which it relates with the classical $\K$, 
%both from a semantical and from a proof-theoretical perspective. %point of view.
%We now illustrate this relation.

Semantics for \intu\ modal logics are typically defined 
by combining \intu\ Kripke models and
possible-world models for modal logics.
%In this case a 
A crucial requirement is that the resulting models must 
preserve the hereditary property of \intu\ %Kripke 
models,
meaning that %if a world satisfies a formula $A$, then 
if a formula is true in a world $w$, then it is true also in all worlds reachable from $w$ 
through the \intu\ order $\less$.
Such a requirement can be fulfilled essentially in two ways.
First, 
one can establish %require 
suitable combinations between %the intuitionistic order 
$\less$ and the modal relation $\R$,
as it is done for instance in the semantics for $\IK$.
%two relations $\R$ and $\less$ to be 
Alternatively, one can build the hereditariness into the satisfaction clauses for modal formulas
by requiring that the standard clauses hold for all $\less$-successors.
This is the strategy adopted by \wij~\cite{wij} who presents models with two relations $\less$ and $\R$ 
without any specific combination between %the two,
them,
where the modalities are interpreted 
%as% % follows:
in the following way:%
\footnote{The semantics of $\CK$ is similar but it also requires 
`fallible' 
worlds
%, that is worlds 
satisfying $\bot$ (cf.~\cite{Mendler1}).}
\begin{center}
\begin{tabular}{lllll}
$\M, w\Vd\Box A$ & iff &  for all $v\more w$, for all $u$, if $v\R u$, then $\M, u\Vd A$. \\
$\M, w\Vd\diam A$ & iff & for all $v\more w$, there is $u$ such that $v\R u$ and $\M, u\Vd A$. \\
\end{tabular}
\end{center}

\wij~\cite{wij}
also provides %an axiomatisation and 
a sequent calculus for $\WK$,
%In particular, the sequent calculus 
which is defined
by extending a suitable %single-succedent %sequent 
calculus for $\IPL$ with the following modal rules
(where $|\G|\geq0$ and $0\leq|\D|\leq1$):
%, for details on the calculus see Sec.~\ref{xx}):
%see Sec.~\ref{xx} for details on the calculus):
\begin{center}
\ax{$\G \seq A$}
\llab{\ruleiKbox}
\uinf{$\Box\G \seq \Box A$}
\disp
\qquad
\ax{$\G, A \seq \D$}
\llab{\ruleiKdiam}
\uinf{$\Box\G, \diam A \seq \diam\D$}
\disp
\end{center}

%As it s known, 
Gentzen \cite{Gentzen} showed that, given a suitable sequent calculus for classical logic, 
its restriction to 
single-succedent sequents (i.e., sequents with at most one formula in the consequent)
provides a sequent calculus for \intu\ logic. % (see e.g.~\cite{Troelstra:2000}).
Interestingly, \wij's logic % calculus
can be seen as 
the system
%the analogous restriction of 
obtained by restricting to single-succedent sequents 
a standard sequent calculus for classical $\K$
(formulated with explicit $\Box$ and $\diam$),%
\footnote{\wij\ \cite{wij} considers a multi-succedent calculus for $\IPL$,
however an equivalent calculus can be given by adding \wij's modal rules to a single-succedent calculus
(cf.~\cite{dal:JPL} and Sec.~\ref{sec:calculi} in this paper).}
so that this correspondence is preserved at the modal level.
We then observe that $\WK$ displays a clear and elegant relation with classical $\K$, both
semantically and proof-theoretically:

\begin{itemize}
\item semantically, $\WK$ is obtained simply by incorporating hereditariness into the modal satisfaction clauses of $\K$;
\item proof-theoretically, it is 
%(can be sees as)
obtained by restricting a standard sequent calculus for $\K$ to single-succedent sequents. 
\end{itemize}

Despite its interest, \wij's logic has received significantly less consideration than $\CK$ and $\IK$.
In particular, while alternative semantics and proof systems for $\WK$ have been studied \cite{Wijesekera:2005,Kojima:2012,dal:JPL,dal:tableaux},
no systematic 
 investigation of \wij-style systems
%, analogous to the one of \const\ and \intu\ systems, 
has been carried out so far.

Filling this gap is precisely the aim of this paper:
we define a family of \wij-style logics
corresponding each to a different classical modal logic
(for lack of a better name we call them \emph{W-logics}),
%We adopt as a guideline for the definition of W-logics 
%the semantical and proof-theoretical relation
%between $\WK$ and $\K$ described above:
%we adopt 
adopting as a guideline for 
%this definition 
the definition of these systems
the semantical and proof-theoretical relation
between $\WK$ and $\K$ just described. % described above.
%The guideline for the definition of W-logics is precisely the semantical and proof-theoretical relation
%between $\WK$ and $\K$ described above:
%we will start by considering standard sequent calculi and semantics for a family of %standard
%classical modal logics, including also $\K$.
%%The semantical and proof-theoretical relation
%%between $\WK$ and $\K$ described above representes the guideline for our definition of W-logics.
%Then we will define \const\ counterparts of these logics by
%(i) restricting the calculi to single-succedent sequents, and
%(ii) expressing the classical satisfaction clauses for modal formulas over \intu\ Kripke models,
%building hereditariness into these conditions. %them.
%We consider
In particular,
in Sec.~\ref{sec:prel} we present standard 
sequent calculi and semantics for a family of  %standard
classical modal logics. %, including also $\K$.
%The semantical and proof-theoretical relation
%between $\WK$ and $\K$ described above representes the guideline for our definition of W-logics.
Then we define \const\ counterparts of these logics by
(i) restricting the calculi to single-succedent sequents (Sec.~\ref{sec:calculi}), and
(ii) expressing the classical satisfaction clauses for modal formulas over \intu\ Kripke models,
building hereditariness into these conditions (Sec.~\ref{sec:semantics}). %them.
%We show that, 
The main contribution of this paper consists in showing that,
despite being
mutually independent,
% independent from each other, 
for a wide family of classical modal logics the semantical and the proof-theoretical approach
return exactly the same \const\ systems.

\section{Preliminaries on classical modal logics}\label{sec:prel}
%We consider a propositional modal language $\lan$
%tWe let 
Let $\lan$ be a propositional modal language
based on a set 
%$\atm = \{p_1, p_2, p_3, ... \}$ of countably many propositional variables;
$\atm$ of countably many propositional variables $p_1, p_2, p_3, ...$;
the \emph{well-formed formulas} of $\lan$ are generated by the following grammar, where $p_i$ is any element of $\atm$:
\begin{center}
$A ::= p_i \mid \bot \mid A \land A \mid A \lor A \mid A \imp A \mid \Box A \mid \diam A$.
\end{center}
%
%Moreover we 
We also 
define $\top := \bot\imp\bot$, $\neg A :=A \imp \bot$,    
and 
$A \tto B := (A \imp B) \land (B \imp A)$.

\begin{figure}
\centering
\begin{tabular}{llllllll}
\multirow{2}{*}{\ax{$A$}
\llab{\nec}
\uinf{$\Box A$}
\disp} & \axKbox &  $\Box(A \imp B) \imp (\Box A \imp \Box B)$ &  \axTbox & $\Box A \imp A$ \\
 & \axKdiam &  $\Box(A \imp B) \imp (\diam A \imp \diam B)$ \ & \axTdiam & $A \imp \diam A$ \\ 
\multirow{2}{*}{\ax{$A \imp B$}
\llab{\monbox}
\uinf{$\Box A \imp \Box B$}
\disp} & \axCbox &  $\Box A\land \Box B \imp \Box(A\land B)$ &  \axD & $\Box A \imp \diam A$ \\
 & \axCdiam &  $\diam(A \lor B) \imp \diam A \lor \diam B$ & \axPbox & $\neg\Box\bot$ \\
\multirow{2}{*}{\ax{$A \imp B$}
\llab{\mondiam}
\uinf{$\diam A \imp \diam B$}
\disp} & \axNbox & $\Box\top$ & \axPdiam & $\diam\top$ \\ 
\vspace{0.1cm}
 & \axNdiam & $\neg\diam\bot$\\
\multicolumn{5}{l}{\  \axdual \quad  $\Box A \tto \neg\diam\neg A$  \hfill \axdualand \quad $\neg(\Box A \land \diam \neg A)$
\hfill \axdualor \quad $\Box A \lor  \diam \neg A$ \ } \\
\end{tabular}

\caption{\label{fig:axioms}Modal axioms and rules.}
\end{figure}

%\begin{figure}
%\begin{tabular}{llllllll}
%\vspace{0.1cm}
%&& \axdual &  $\Box A \tto \neg\diam\neg A$ \\
%\multirow{2}{*}{\ax{$A \imp B$}
%\llab{\monbox}
%\uinf{$\Box A \imp \Box B$}
%\disp} 
% &&
%\axCbox &  $\Box A\land \Box B \imp \Box(A\land B)$ && \axTbox & $\Box A \imp A$ \\
%\vspace{0.1cm}
%&&
%\axNbox &  $\Box \top$ && \axD & $\Box A \imp \diam A$ \\
%\multirow{2}{*}{\ax{$A \imp B$}
%\llab{\mondiam}
%\uinf{$\diam A \imp \diam B$}
%\disp} 
% &&
%\axCbox &  $\diam(A\lor B)\imp \diam A \lor\diam B$ && \axTdiam & $A \imp \diam A$ \\
%&&
%\axNdiam & $\neg\diam\bot$ \\
%\end{tabular}
%\caption{\label{fig:axioms}Modal axioms and rules.}
%\end{figure}

We aim at %extending
enriching the family of \wij-style propositional modal logics by defining 
%a corresponding constructive system for each of of a set of well-known classical modal logics. 
%constructive systems corresponding to well-known classical modal logics. 
constructive counterparts of well-known classical modal logics. 
%We consider extensions of $\K$ but also weaker (non-normal) logics.
We consider the following classical systems, which are 
defined in the language $\lan$ extending 
(any axiomatisation of) classical propositional logic ($\CPL$) 
with the following modal axioms and rules from Fig.~\ref{fig:axioms}:
%and contain extensions of $\K$ as well as weaker (i.e., non-normal) systems:
%
\begin{center}
%% SENZA AXP
%\begin{tabular}{lllll}
%$\EM$ := \axdual\ + \monbox &&  $\EMD$ := $\EM$ + \axD &&  $\EMT$ := $\EM$ + \axTbox \\
%$\EMN$ := $\EM$ + \axNbox && $\EMND$ := $\EMN$ + \axD &&  $\EMNT$ := $\EMN$ + \axTbox \\
%$\EMC$ := $\EM$ + \axCbox && $\EMCD$ := $\EMC$ + \axD &&  $\EMCT$ := $\EMC$ + \axTbox \\
%$\K$ := $\EM$ + \axNbox\ + \axCbox && $\KD$ := $\K$ + \axD &&  $\KT$ := $\K$ + \axTbox \\
%\end{tabular}
%
%\begin{tabular}{lllll}
%$\EM$ := \axdual\ + \monbox &&  $\EMD$ := $\EM$ + \axD &&  $\EMCD$ := $\EMC$ + \axD \\
%$\EMN$ := $\EM$ + \axNbox && $\EMT$ := $\EM$ + \axTbox &&  $\EMCT$ := $\EMC$ + \axTbox \\  
%$\EMC$ := $\EM$ + \axCbox &&  $\EMNP$ := $\EMN$ + \axPbox &&  $\KD$ := $\K$ + \axD \\
%$\K$ := $\EM$ + \axNbox\ + \axCbox && $\EMND$ := $\EMN$ + \axD &&  $\KT$ := $\K$ + \axTbox \\
%$\EMP$ := $\EM$ + \axPbox && $\EMNT$ := $\EMN$ + \axTbox \\
%\end{tabular}
%
\begin{tabular}{lllll}
$\EM$ := \axdual\ + \monbox &&  $\EMNP$ := $\EMN$ + \axPbox &&  $\EMT$ := $\EM$ + \axTbox \\
$\EMN$ := $\EM$ + \axNbox && $\EMD$ := $\EM$ + \axD &&  $\EMNT$ := $\EMN$ + \axTbox \\  
$\EMC$ := $\EM$ + \axCbox &&  $\EMND$ := $\EMN$ + \axD &&  $\EMCT$ := $\EMC$ + \axTbox \\ 
$\K$ := $\EM$ + \axNbox\ + \axCbox && $\EMCD$ := $\EMC$ + \axD &&  $\KT$ := $\K$ + \axTbox \\
$\EMP$ := $\EM$ + \axPbox && $\KD$ := $\K$ + \axD \\
\end{tabular}
\end{center}
The considered axiomatisation of $\K$ is equivalent to the more standard one with \nec\ and \axKbox\ (cf.~e.g.~\cite{Chellas:1980}).
The above list contains %extensions of $\K$
logics %\tiz{which are} 
stronger than $\K$ as well as weaker (i.e., non-normal) systems. %logics.
Note that given the duality between $\Box$ and $\diam$, % in classical logics, 
the above systems can be equivalently defined by
replacing \monbox, \axNbox, \axCbox, \axTbox, and \axPbox, with their %$\diam$ counterparts
$\diam$-versions \mondiam, \axNdiam, \axCdiam, \axTdiam, \axPdiam\ %, respectively
 (Fig.~\ref{fig:axioms}).
The relations among the classical systems are displayed in Fig.~\ref{fig:cl dyag}
($\EMCP$ and $\KP$ are not considered in the list as they coincide with $\EMCD$ and $\KD$).

%\begin{figure}
%\centering
%\begin{small}
%%\begin{footnotesize}
%\begin{tikzpicture}
%    \node (M) at  (0,0)  {$\EM$};
%    \node (MN) at (1, 0.9) {$\EMN$};
%    \node  (MC) at (1, -0.9) {$\EMC$};
%    \node (K) at (2, 0) {$\K$};
%
%    \node (MP) at  (3.4,0)  {$\EMP$};
%    \node (MNP) at (4.4, 0.9) {$\EMNP$};
%
%    \node (MD) at  (5.4,0)  {$\EMD$};
%    \node (MND) at (6.4, 0.9) {$\EMND$};
%    \node  (MCD) at (6.4, -0.9) {$\EMCD$};
%    \node (KD) at (7.4, 0) {$\KD$};
%
%    \node (MT) at  (8.6,0)  {$\EMT$};
%    \node (MNT) at (9.6, 0.9) {$\EMNT$};
%    \node  (MCT) at (9.6, -0.9) {$\EMCT$};
%    \node (KT) at (10.6, 0) {$\KT$};
%
%	\draw[->] (M) -- (MN);
%	\draw[->] (M) -- (MC);
%	\draw[->] (MN) -- (K);	
%	\draw[->] (MC) -- (K);
%	\draw[->] (MP) -- (MNP);
%	\draw[->] (MD) -- (MND);
%	\draw[->] (MD) -- (MCD);
%	\draw[->] (MND) -- (KD);	
%	\draw[->] (MCD) -- (KD);
%	\draw[->] (MT) -- (MNT);
%	\draw[->] (MT) -- (MCT);
%	\draw[->] (MNT) -- (KT);	
%	\draw[->] (MCT) -- (KT);
%
%	\draw[->] (MN) -- (MNP);
%	\draw[->] (MNP) -- (MND);
%	\draw[->] (MND) -- (MNT);
%	\draw[->] (MC) -- (MCD);
%	\draw[->] (MCD) -- (MCT);
%	\draw[->] (MP) -- (MD);
%
%	\path[->] (M) edge [bend left=25] (MP);
%	\path[->] (MD) edge [bend left=25] (MT);
%	\path[->] (KD) edge [bend right=25] (KT);
%	\path[->] (K) edge [bend right=18] (KD);
%\end{tikzpicture}
%\end{small}
%%\end{footnotesize}
%\caption{\label{fig:cl dyag} Dyagram of classical modal logics.}
%\end{figure}

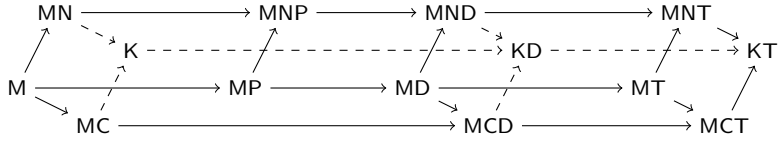
\begin{figure}
\centering
%\begin{small}
\begin{footnotesize}
\begin{tikzpicture}
%    \node (M) at  (0,0)  {$\EM$};
%    \node (MN) at (0.6, 1.2) {$\EMN$};
%    \node  (MC) at (1.2, -0.6) {$\EMC$};
%    \node (K) at (1.8, 0.6) {$\K$};
%
%    \node (MP) at  (3.3,0)  {$\EMP$};
%    \node (MNP) at (3.9, 1.2) {$\EMNP$};
%
%    \node (MD) at  (5.4,0)  {$\EMD$};
%    \node (MND) at (6, 1.2) {$\EMND$};
%    \node  (MCD) at (6.6, -0.6) {$\EMCD$};
%    \node (KD) at (7.2, 0.6) {$\KD$};
%
%    \node (MT) at  (8.6,0)  {$\EMT$};
%    \node (MNT) at (9.2, 1.2) {$\EMNT$};
%    \node  (MCT) at (9.8, -0.6) {$\EMCT$};
%    \node (KT) at (10.4, 0.6) {$\KT$};
    \node (M) at  (0,0)  {$\EM$};
    \node (MN) at (0.5, 1) {$\EMN$};
    \node  (MC) at (1, -0.5) {$\EMC$};
    \node (K) at (1.5, 0.5) {$\K$};

    \node (MP) at  (3.,0)  {$\EMP$};
    \node (MNP) at (3.5, 1) {$\EMNP$};

    \node (MD) at  (5.2,0)  {$\EMD$};
    \node (MND) at (5.7, 1) {$\EMND$};
    \node  (MCD) at (6.2, -0.5) {$\EMCD$};
    \node (KD) at (6.7, 0.5) {$\KD$};

    \node (MT) at  (8.3,0)  {$\EMT$};
    \node (MNT) at (8.8, 1) {$\EMNT$};
    \node  (MCT) at (9.3, -0.5) {$\EMCT$};
    \node (KT) at (9.8, 0.5) {$\KT$};

	\draw[->] (M) -- (MN);
	\draw[->] (M) -- (MC);
	\draw[->, dashed] (MN) -- (K);	
	\draw[->, dashed] (MC) -- (K);
	\draw[->] (MP) -- (MNP);
	\draw[->] (MD) -- (MND);
	\draw[->] (MD) -- (MCD);
	\draw[->, dashed] (MND) -- (KD);	
	\draw[->, dashed] (MCD) -- (KD);
	\draw[->] (MT) -- (MNT);
	\draw[->] (MT) -- (MCT);
	\draw[->] (MNT) -- (KT);	
	\draw[->] (MCT) -- (KT);

	\draw[->] (MN) -- (MNP);
	\draw[->] (MNP) -- (MND);
	\draw[->] (MND) -- (MNT);
	\draw[->] (MC) -- (MCD);
	\draw[->] (MCD) -- (MCT);
	\draw[->] (MP) -- (MD);

	\draw[->] (M) -- (MP);
	\draw[->] (MD) -- (MT);
	\draw[->, dashed] (KD) -- (KT);
	\draw[->, dashed] (K) -- (KD);
\end{tikzpicture}
%\end{small}
\end{footnotesize}
\caption{\label{fig:cl dyag} Dyagram of classical modal logics.}
\end{figure}

%In this paper we define intuitionistic counterparts of classical modal logics
%by %suitably restricting
%considering natural restrictions of sequent calculi for the 
%%latter kind of systems.
%classical systems.
We will define \const\ counterparts of classical modal logics
by restricting suitable sequent calculi for the classical systems.
%to having at most one formula in the consequent.
We consider to this purpose
the %sequent 
calculi for classical modal logics
defined by the rules in Fig.~\ref{fig:classical sequent calculi}.
As usual, we call \emph{sequent} any pair $\G\seq\D$, where $\G$ and $\D$ are finite, possibly empty multisets of formulas of $\lan$. 
A sequent $\G\seq\D$ is interpreted as a formula of $\lan$ as $\AND\G \imp \OR\D$
if $\G$ is non-empty, and it is interpreted as $\OR\D$ if $\G$ is empty, where
$\OR\emptyset$ is interpreted as $\bot$.
%If $\G$ is a multiset $A_1, ..., A_n$,
%Moreover, for
For every multiset $\G = A_1, ..., A_n$,
 we denote with $\Box\G$ and $\diam\G$ the multisets
$\Box A_1, ..., \Box A_n$ and $\diam A_1, ..., \diam A_n$, respectively. 
%\tiz{The %considered 
We consider
%calculi in the form of G3-systems (cf.~\cite[Ch.~3]{Troelstra:2000}),
%\tiz{\ie, with side contexts in the conclusions of the modal rules}. 
G3-style calculi with all structural rules admissible (cf.~\cite[Ch.~3]{Troelstra:2000}).
% in order to embed weakening in their application}.
%%
%\footnote{The results presented in this work could be also based on calculi with all structural rules admissible,
%such as G3-calculi. We opted for G1-calculi as they offer a more direct relation between 
%classical and \intu\ systems.}
%The calculi are formulated with explicit $\Box$ and $\diam$,
%The rules are formulated for both $\Box$ and $\diam$,
%We consider a formulation of the rules in which both $\Box$ and $\diam$ occur,
%We 
Moreover, we consider a formulation of the calculi in which both $\Box$ and $\diam$ occur explicitly,
%(cf.~e.g.~\cite[Ch.~9]{Troelstra:2000}),
%this is needed in order to capture intuitionistic modal logics by restricting the calculi,
%since the modalities are not interdefinable in the intuitionistic systems.
this formulation will be
 needed %in order to capture constructive modal logics, 
%for 
to handle the constructive systems,
%in which 
where the modalities are not interdefinable
(for a sequent calculus with explicit $\Box$ and $\diam$ see e.g.~\cite[Ch.~9]{Troelstra:2000}, for sequent calculi for non-normal modal logics see \cite{Lavendhomme:2000,Lellmann:2019}).
For each logic $\logic$, the corresponding calculus $\seqlogic$ %is defined as follows:
contains the propositional rules %, the structural rules, 
and the following modal rules: %\nb{Aggiungere regole P a G3MD e G3MND}
%
%%VECCHIA VERSIONE (CASO T AMBIGUO)
%\begin{center}
%\begin{tabular}{ll}
%$\seqEM$ := \ruleMbox\ + \ruleMdiam\ + \ruledualandM\ + \ruledualorM & 
%$\seqEMN$ := $\seqEM$ + \ruleNbox\ + \ruleNdiam \\ 
%$\seqEMC$ := \ruleCbox\ + \ruleCdiam\ + \ruledualandC\ + \ruledualorC &
%$\seqEMP$ := $\seqEM$ + \rulePbox\ + \rulePdiam \\ 
%$\seqEMNP$ := $\seqEMN$ + \rulePbox\ + \rulePdiam & 
%$\seqEMCD$ := $\seqEMC$ + \ruleCD \\
%$\seqEMD$ := $\seqEM$ + \ruleD\ + \ruleDbox\ + \ruleDdiam &
%$\seqK$ := \ruleKbox\ + \ruleKdiam \\
%$\seqEMND$ := $\seqEMN$ + \ruleD\ + \ruleDbox\ + \ruleDdiam &
%$\seqKD$ := $\seqK$ + \ruleCD \\
%\end{tabular}
%\end{center}
%%
%Moreover, for every calculus $\seqlogic$ above,
%$\Gone{LT}$ := $\seqlogic$ + \ruleTbox\ + \ruleTdiam
%%FINE VECCIA VERSIONE
\begin{center}
\begin{small}
\begin{tabular}{ll}
$\seqEM$ := \ruleMbox\ + \ruleMdiam\ + \ruledualandM\ + \ruledualorM &
$\seqEMP$ := $\seqEM$ + \rulePbox\ + \rulePdiam \\
$\seqEMN$ := $\seqEM$ + \ruleNbox\ + \ruleNdiam &
$\seqEMNP$ := $\seqEMN$ + \rulePbox\ + \rulePdiam \\
$\seqEMC$ := \ruleCbox\ + \ruleCdiam\ + \ruledualandC\ + \ruledualorC \\
\vspace{0.1cm}
$\seqK$ := \ruleKbox\ + \ruleKdiam \\
$\seqEMD$ := $\seqEM$ + \ruleD\ + \ruleDbox\ + \ruleDdiam\ + \rulePbox\ + \rulePdiam &
$\seqEMT$ := $\seqEM$ + \ruleTbox\ + \ruleTdiam \\
$\seqEMND$ := $\seqEMN$ + \ruleD\ + \ruleDbox\ + \ruleDdiam\ + \rulePbox\ + \rulePdiam &
$\seqEMNT$ := $\seqEMN$ + \ruleTbox\ + \ruleTdiam \\
$\seqEMCD$ := $\seqEMC$ + \ruleCD &
$\seqEMCT$ := $\seqEMC$ + \ruleTbox\ + \ruleTdiam \\
$\seqKD$ := $\seqK$ + \ruleCD &
$\seqKT$ := $\seqK$ + \ruleTbox\ + \ruleTdiam \\
\end{tabular}
\end{small}
\end{center}

\begin{figure}
\begin{small}
\centering
%\textbf{Propositional rules} \hfill \quad
%
%\init\ \  $\G, p \seq p, \D$ \qquad\qquad \lbot\ \  $\G, \bot\seq \D$
\textbf{Propositional rules} 
\hfill 
\init\ \  $\G, p \seq p, \D$ 
\hfill
\lbot\ \  $\G, \bot\seq \D$

\vspace{0.3cm}
\ax{$\G \seq A, \D$}
\ax{$\G, B \seq \D$}
\llab{\limp}
\binf{$\G, A \imp B \seq \D$}
\disp
\hfill
\ax{$\G, A\seq B, \D$}
\llab{\rimp}
\uinf{$\G \seq A\imp B, \D$}
\disp
\hfill
\ax{$\G \seq A, \D$}
\ax{$\G \seq B, \D$}
\llab{\rland}
\binf{$\G \seq A\land B, \D$}
\disp

\vspace{0.2cm}
\ax{$\G, A, B \seq \D$}
\llab{\lland}
\uinf{$\G, A\land B \seq \D$}
\disp
\hfill
\ax{$\G \seq A, B, \D$}
\llab{\rlor}
\uinf{$\G \seq A\lor B, \D$}
\disp
\hfill
\ax{$\G, A \seq \D$}
\ax{$\G, B \seq \D$}
\llab{\llor}
\binf{$\G, A \lor B \seq \D$}
\disp

%\vspace{0.3cm}
%\textbf{Structural rules} \hfill \quad
%
%\vspace{0.2cm}
%\ax{$\G \seq \D$}
%\llab{\lwk}
%\uinf{$\G, A \seq \D$}
%\disp
%\hfill
%\ax{$\G \seq \D$}
%\llab{\rwk}
%\uinf{$\G \seq A, \D$}
%\disp
%\hfill
%\ax{$\G, A, A\seq \D$}
%\llab{\lctr}
%\uinf{$\G, A \seq \D$}
%\disp
%\hfill
%\ax{$\G \seq A, A, \D$}
%\llab{\rctr}
%\uinf{$\G \seq A, \D$}
%\disp

\vspace{0.3cm}
\textbf{Modal rules} \hfill \quad

\vspace{0.2cm}
\ax{$A \seq B$}
\llab{\ruleMbox}
\uinf{$\G, \Box A \seq \Box B, \D$}
\disp
\hfill
\ax{$A \seq B$}
\llab{\ruleMdiam}
\uinf{$\G, \diam A \seq \diam B, \D$}
\disp
\hfill
\ax{$\G, A \seq B, \D$}
\llab{\ruleCbox}
\uinf{$\G', \Box\G, \Box A \seq \Box B,  \diam\D, \D'$}
\disp

\vspace{0.2cm}
\ax{$A, B \seq$}
\llab{\ruledualandM}
\uinf{$\G, \Box A, \diam B \seq \D$}
\disp
\hfill
\ax{$\G, A \seq B, \D$}
\llab{\ruleCdiam}
\uinf{$\G', \Box\G, \diam A \seq \diam B, \diam\D, \D'$}
\disp
\hfill
\ax{$A \seq B$}
\llab{\ruleD}
\uinf{$\G, \Box A \seq \diam B, \D$}
\disp

\vspace{0.2cm}
\ax{$\seq A, B$}
\llab{\ruledualorM}
\uinf{$\G \seq \Box A,  \diam B, \D$}
\disp
\hfill
\ax{$\G, A, B \seq$}
\llab{\ruledualandC}
\uinf{$\G', \Box\G, \Box A, \diam B \seq \D$}
\disp
\hfill
\ax{$A, B \seq$}
\llab{\ruleDbox}
\uinf{$\G, \Box A, \Box B \seq \D$}
\disp

\vspace{0.2cm}
\ax{$\seq A, B, \D$}
\llab{\ruledualorC}
\uinf{$\G \seq \Box A,  \diam B, \diam\D, \D'$}
\disp
\hfill
\ax{$\G \seq A, \D$}
\llab{\ruleKbox}
\uinf{$\G', \Box\G \seq \Box A,  \diam\D, \D'$}
\disp
\hfill
\ax{$\seq A, B$}
\llab{\ruleDdiam}
\uinf{$\G \seq \diam A, \diam B, \D$}
\disp

\vspace{0.2cm}
\ax{$\G, A \seq \D$}
\llab{\ruleKdiam}
\uinf{$\G', \Box\G, \diam A \seq \diam\D, \D'$}
\disp
\hfill
\ax{$\seq A$}
\llab{\ruleNbox}
\uinf{$\G \seq \Box A, \D$}
\disp
\hfill
\ax{$A \seq$}
\llab{\ruleNdiam}
\uinf{$\G, \diam A \seq \D$}
\disp
\hfill
\ax{$A \seq$}
\llab{\rulePbox}
\uinf{$\G, \Box A \seq \D$}
\disp

\vspace{0.2cm}
\ax{$\seq A$}
\llab{\rulePdiam}
\uinf{$\G \seq\diam A, \D$}
\disp
\hfill
\ax{$\G, \Box A, A \seq \D$}
\llab{\ruleTbox}
\uinf{$\G, \Box A \seq \D$}
\disp
\hfill
\ax{$\G \seq A, \diam A, \D$}
\llab{\ruleTdiam}
\uinf{$\G \seq \diam A, \D$}
\disp
\hfill
\ax{$\G \seq \D$}
\llab{\ruleCD}
\uinf{$\G', \Box\G \seq \diam\D, \D'$}
\disp

\end{small}
\caption{\label{fig:classical sequent calculi} 
%Sequent calculus for classical $\K$
Sequent rules for classical modal logics
%(where $|\G|\geq 0$, $|\D|\geq 0$, $i \in\{1, 2\}$).}
(where $|\G|, |\G'|, |\D|, |\D'| \geq 0$).}
\end{figure}

%\tiz{NB: D così non va bene: copre anche P nel caso non normale}

%\vspace{0.2cm}
%\ax{$\G, A \seq B, \D$}
%\llab{\ruleRbox}
%\uinf{$\Box\G, \Box A \seq \Box B,  \diam\D$}
%\disp
%\hfill
%\ax{$\G, A \seq B, \D$}
%\llab{\ruleRdiam}
%\uinf{$\Box\G, \diam A \seq \diam B, \diam\D$}
%\disp

Each calculus contains two duality rules $\land$-dual and $\lor$-dual
(in $\seqK$ and its extensions they are obtained as the particular cases of \ruleKdiam\ and \ruleKbox\
with 
%$|\D|=\emptyset$, 
$\D=\emptyset$,
respectively 
%$|\G|=\emptyset$).
$\G=\emptyset$).\label{correction1}
%The duality rules allow one to derive the 
%axioms
%\begin{center}
%\axdualand \quad $\neg(\Box A \land \diam \neg A)$
%\qquad\quad
%\axdualor \quad $\Box A \lor  \diam \neg A$
%\end{center}
%that taken together are equivalent to \axdual.
%We call them respectively 
%\emph{conjunctive} and \emph{disjunctive} duality.
The duality rules 
allow one to derive
the
Hilbert-style rules
\begin{center}
\ax{$\neg(A \land B)$}
\llab{\Rdualand}
\uinf{$\neg(\Box A \land \diam B)$}
\disp
\qquad\quad
\ax{$A \lor B$}
\llab{\Rdualor}
\uinf{$\Box A \lor \diam B$}
\disp
\end{center}
which are classically equivalent to \axdualand\ and \axdualor\ (Fig.~\ref{fig:axioms}),
and taken together are equivalent to \axdual.
%We call them respectively 
%\emph{conjunctive} and \emph{disjunctive} duality.
%
The rules \ruleCbox\ and \ruleCdiam\ can be seen as the generalisation of \ruleMbox\ and \ruleMdiam\
to $n$-principal formulas in the antecedent, respectively in the consequent.
Differently from \ruleMbox, the rule \ruleCbox\ involves also $\diam$-formulas,
similarly \ruleCdiam\ involves also $\Box$-formulas,
%these generalisations involve also $\diam$-, respectively $\Box$-formulas, 
this is needed in order to preserve the admissibility of cut in the calculus.
Note also that %the rules 
\ruleCbox\ and \ruleCdiam\ are distinct from 
\ruleKbox\ and \ruleKdiam,
since \ruleCbox\ and \ruleCdiam\ are applicable only to sequents with non-empty antecedent, respectively non-empty consequent, 
while 
this is not required for \ruleKbox\ and \ruleKdiam.
Finally, 
%%observe %also 
%%that 
%%the rules \rulePbox\ and \rulePdiam\ are derivable in 
%%any calculus containing the rules for \axD:
%%in particular for $\seqEMD$,
%%from $\seq A$, by \rwk\ we have $\seq A,A$, then by \ruleDdiam, $\seq \diam A, \diam A$, finally by \rctr, $\seq \diam A$,
%%with a similar derivation we obtain \rulePbox.
%in the calculi 
%$\seqEMD$ and $\seqEMND$
%we also require the rules \rulePbox\ and \rulePdiam\
%in order to ensure admissibility of contraction \cite{dal:JLC,Orlandelli},
%this is consistent with the fact that \axPbox\ and \axPdiam\ are derivable in $\EMD$.
the calculi $\seqEMD$ and $\seqEMND$
contain also the rules \rulePbox\ and \rulePdiam, this is needed to ensure admissibility of contraction \cite{dal:JLC,Orlandelli},
and is consistent with the fact that the axioms \axPbox\ and \axPdiam\ are derivable in $\EMD$.
%Note also that 
%the rules \rulePbox\ and \rulePdiam\ are derivable in 
%%$\seqEMD$:
%any calculus containing the rules for \axD:
%for $\seqEMD$,
%%from $A\seq$, by \lwk\ we have $A,A\seq$, then by \ruleDbox, $\Box A, \Box A \seq$, finally by \lctr, $\Box A \seq$;
%%similarly for \rulePdiam.}
%from $\seq A$, by \rwk\ we have $\seq A,A$, then by \ruleDdiam, $\seq \diam A, \diam A$, finally by \rctr, $\seq \diam A$;
%%similarly for \rulePbox.
%with a similar derivation we obtain \rulePbox.
%This is consistent with the fact that \axPbox\ and \axPdiam\ are derivable in $\EMD$.
%$\seqEMD \vd$ \rulePbox, \rulePdiam.}
%
Each calculus $\seqlogic$ is a calculus for the corresponding logic $\logic$ in the following sense:

%The following is a standard result \tiz{[dare ref.?]}.

%The calculus $\seqK$ for $\K$ contains the propositional rules and the rules \ruleKbox\ and \ruleKdiam, 
%while the calculi for the extensions %of $\K$
% with the axioms \axTbox, or \axfourbox, or \axD,
%also contain the corresponding rules \ruleTbox\ and \ruleTdiam, or \rulefourbox\ and \rulefourdiam,
%or \ruleD. The calculi for the systems %containing 
%with both \axfourbox\ and \axD\ also contain the rule
%\ruleDfour, which is needed in order to preserve the admissibility of cut.
%% in order to satisfy admissibility of cut.

%The behaviour of the classical $\K$-modalities is captured in $\seqK$ by means of the rules \ruleKbox\ and \ruleKdiam.
%Note that as particular cases of these rules we also have
%\begin{center}
%\ax{$A, B \seq$}
%\llab{\ruledualand}
%\uinf{$\Box A, \diam B \seq $}
%\disp
%\qquad\quad
%\ax{$\seq A, B$}
%\llab{\ruledualor}
%\uinf{$\seq \Box A,  \diam B$}
%\disp
%\end{center}
%which express the duality between $\Box$ and $\diam$.
%%By means of these rules we can derive the axioms
%The duality rules allow one to derive the following axioms
%\begin{center}
%\axdualand \quad $\neg(\Box A \land \diam \neg A)$
%\qquad\quad
%\axdualor \quad $\Box A \lor  \diam \neg A$
%\end{center}
%that we call respectively 
%%the \emph{conjunctive} and the \emph{disjunctive} side of duality.
%\emph{conjunctive} and \emph{disjunctive} duality.

\begin{theorem}
For every considered 
classical modal
logic $\logic$, $\seqlogic\vd \G\seq\D$ if and only if $\logic\vd\AND\G\imp\OR\D$. 
\end{theorem}

%Finally, we consider a neighbourhood semantics that uniformly covers all considered systems
%(non-normal logics do not have an equally simple relational semantics).

We now move to the semantics.
Since non-normal logics do not have a (simple) relational semantics,%
\footnote{Cf.~\cite{Calardo:2014} for multi-relational semantics for non-normal modal logics, and
\cite{Priest:2008,dal:JLC} for relational semantics with ``non-normal'' worlds for the logics %non-normal logics
%with axiom \axCbox.}
containing \axCbox\ but not \axNbox.}
we consider a neighbourhood semantics that uniformly covers all considered %(normal and non-normal) 
systems.

\begin{definition}\label{def:neigh model}
A \emph{neighbourhood model} is a tuple $\M = \langle \W, \N, \V\rangle$, 
where 
$\W$ is a non-empty set of worlds, 
$\N$ is a function $\W\longrightarrow \pow(\pow(\W))$,
called neighbourhood function, and
$\V$ is a valuation function $\atm \longrightarrow \pow(\W)$.
%$\V : \atm \longrightarrow \pow(\W)$ is a valuation function
%for the propositional varialbes of $\lan$.
The forcing relation $\M, w\Vd A$ is 
%is defined as in Def.~\ref{def:wij model} for $A = p, \bot, B \land C, B \lor C$, otherwise it is as follows:
inductively defined as follows:
\begin{center}
\begin{tabular}{lllll}
$\M, w \Vd p$ & iff & $w\in\V(p)$. \\
$\M, w \not\Vd \bot$. \\
$\M, w \Vd B \land C$ & iff & $\M, w \Vd B$ and $\M, w \Vd C$. \\
$\M, w \Vd B \lor C$ & iff & $\M, w \Vd B$ or $\M, w \Vd C$. \\
$\M, w \Vd B \imp C$ & iff & $\M, w \not\Vd B$ or $\M, w \Vd C$. \\
$\M, w\Vd\Box B$ & iff & there is $\alpha\in\N(w)$ s.t.~for all $v\in\alpha$, $\M, v\Vd B$. \\
$\M, w\Vd\diam B$ & iff & for all $\alpha\in\N(w)$, there is $v\in\alpha$ s.t.~$\M, v\Vd B$. \\ 
\end{tabular}
\end{center}
We consider the following properties on neighbourhood models:
\begin{center}
%% BUONO SENZA AXP
%\begin{tabular}{llll}
%(\cC) \ If $\alpha,\beta\in\N(w)$, then $\alpha\cap\beta\in\N(w)$. &
%(\cN) \ $\N(w)\not=\emptyset$. \\
%(\cD) \  If $\alpha,\beta\in\N(w)$, then $\alpha\cap\beta\not=\emptyset$. &
%(\cT) \ If $\alpha\in\N(w)$, then $w\in\alpha$. \\
%\end{tabular}
\begin{tabular}{llll}
(\cC) \ If $\alpha,\beta\in\N(w)$, then $\alpha\cap\beta\in\N(w)$. &&
(\cN) \ $\N(w)\not=\emptyset$. \\
(\cD) \  If $\alpha,\beta\in\N(w)$, then $\alpha\cap\beta\not=\emptyset$. &&
(\cP) \  $\emptyset\notin\N(w)$. \\
(\cT) \ If $\alpha\in\N(w)$, then $w\in\alpha$. \\
\end{tabular}
\end{center}
We say that $\M$ is a model for a logic $\logic$ if it satisfies the condition (\cX) for every modal 
axiom \varaxbox\ of $\logic$
(among \axCbox, \axNbox, \axTbox, \axD, \axPbox).
As usual, we say that a formula $A$ is valid in a model 
$\M$, written $\M\models A$, if $\M,w\Vd A$ for every world $w$ of $\M$.
\end{definition}

In the following we simple write $w\Vd A$ when $\M$ is clear from the context.
We also use the following abbreviations:
\begin{center}
$\alpha\ufor A$ := for all $w\in\alpha$, $w\Vd A$; \qquad
$\alpha\efor A$ := there is $w\in\alpha$ s.t.~$w\Vd A$. 
\end{center}
Using these abbreviations,
the 
%forcing conditions
satisfaction clauses
for modal formulas can be equivalently written as
\begin{center}
\begin{tabular}{lllll}
$w\Vd\Box B$ & iff & there is $\alpha\in\N(w)$ such that $\alpha\ufor B$. \\
$w\Vd\diam B$ & iff & for all $\alpha\in\N(w)$, $\alpha\efor B$. \\ 
\end{tabular}
\end{center}

The following holds (cf.~\eg~\cite{Chellas:1980,Pacuit:2017}).

\begin{theorem}
For every considered 
classical modal
logic $\logic$, $\logic\vd A$ if and only if $\M\models A$ for all neighbourhood models $\M$ for $\logic$. 
\end{theorem}

%%% RELATIONAL MODELS FOR REGULAR LOGICS
%\begin{definition}\label{def:minimal model}
%A \emph{Kripke model} with non-normal worlds is a tuple $\M = \langle \W, \NNW, \R, \V\rangle$, 
%where 
%%$\W$ is a non-empty set of items, called worlds, 
%$\W$ is a non-empty set of items called possible worlds, 
%$\NNW \subseteq\W$,
%$\R$ is a binary relation on $\W$,
%and 
%$\V : \atm \longrightarrow \pow(\W)$ is a valuation function 
%for the propositional variables.
%The forcing relation $\M, w\Vd A$ is inductively defined as follows:
%\begin{center}
%\begin{tabular}{lllll}
%$\M, w \Vd p$ & iff & $w\in\V(p)$; \\
%$\M, w \not\Vd \bot$ \\
%$\M, w \Vd B \land C$ & iff & $\M, w \Vd B$ and $\M, w \Vd C$; \\
%$\M, w \Vd B \lor C$ & iff & $\M, w \Vd B$ or $\M, w \Vd C$; \\
%$\M, w \Vd B \imp C$ & iff & $\M, w \Vd B$ implies $\M, w \Vd C$; \\
%$\M, w \Vd \Box B$ & iff & $w\notin\NNW$ and for all $v\in\W$, $w\R v$ implies $\M, v \Vd B$; \\
%$\M, w \Vd \diam B$ & iff & $w\in\NNW$, or there is $v\in\W$ s.t.~$w\R v$ and $\M, v \Vd B$. \\
%\end{tabular}
%\end{center}
%\end{definition}

\section{Single-succedent calculi and W-logics}\label{sec:calculi} %\const\ modal logics}\label{sec:calculi} %\wij's logics

%We now define a family of \wij-style \const\ modal logics 
%corresponding to the classical logics considered in Sec.~\ref{sec:prel}.
%The definition of these systems is based on the sequent calculi.
%In particular, we require that the \const\ calculi preserve the relation holding between classical $\K$ and \wij's logic,
%to this aim we define the \const\ systems by restricting the calculi for classical modal logics 
%to single-succedent sequents.
%We show that the resulting calculi enjoy admissibility of cut, 
%%we also
%then we present equivalent axiomatic systems.

%Basing on the sequent calculi, w
We now define a family of \wij-style \const\ modal logics 
corresponding to the classical logics considered in Sec.~\ref{sec:prel}.
%%In order to 
%%%ensure that the relation existing between $\K$ and $\WK$ also holds for the other systems,
%%extend to all logics the same relation holding between $\K$ and $\WK$,
%%we define the \const\ logics 
%%by restricting the classical calculi to single-succedent sequents.
%In particular, the \const\ calculi are defined 
%by restricting the classical calculi to single-succedent sequents,
%so that the relation holding between $\K$ and $\WK$
%is extended to the other systems.
%We show that the resulting calculi enjoy admissibility of cut, 
%we also
%%then we 
%present equivalent axiomatic systems.
In particular, we firstly define \const\ modal calculi 
by restricting the classical calculi from Sec.~\ref{sec:prel} to single-succedent sequents,
%\tiz{so that the relation holding between $\K$ and $\WK$ is extended to the other systems,}
and study their structural properties.
Then we define equivalent axiomatic systems,
and prove some fundamental properties of them.
%\textcolor{red}{and prove some fundamental properties of them}.

%%Basing on the sequent calculi, we now define a family of \wij-style \const\ modal logics 
%%We require that the calculi for \const\ logics preserve the relation holding between classical $\K$ and \wij's system,
%%to this aim we define the calculi for \const\ logics by restricting the calculi for classical modal logics 
%%presented in Sec.~\ref{sec:prel}
%%to single-succedent sequents.

%
%%%We now move to the definition of a constructive variant for each logic introduced in Sec.~\ref{sec:prel}.
%%We now define a constructive variant of each logic introduced in Sec.~\ref{sec:prel}.
%%In G1-calculi there is an elegant relation between classical and intuitionistic logic: the calculus $\Gonei$ for $\IPL$
%G1-calculi 
%%allow for an elegant relation 
%offer an elegant connection
%between classical and intuitionistic logic: the calculus $\Gonei$ for $\IPL$ is obtained from $\Gonec$ 
%simply by restricting the sequents to having at most one formula in the consequent.
%In this section we define constructive variants of 
%classical modal logics in an analogous way: 
%%for each classical logic $\logic$ introduced in Sec.~\ref{sec:prel},
%for each classical calculus $\seqlogic$ introduced in Sec.~\ref{sec:prel},
%we define $\seqWlogic$ as the calculus obtained by restricting $\seqlogic$
%to sequents with at most one formula in the consequent.
%%each logic introduced in Sec.~\ref{sec:prel}.
%For each calculus $\seqWlogic$ we show that the cut rule\nb{cut non ancora menzionata} is admissible,
%and also provide an equivalent axiomatization.

\begin{figure}
\begin{small}
\centering
%\textbf{Propositional rules} \hfill \quad
%
%\iinit\ \  $\G, p \seq p$ \qquad\qquad \ilbot\ \  $\G, \bot\seq \D$
\textbf{Propositional rules} 
\hfill 
\iinit\ \  $\G, p \seq p$ 
\hfill
\ilbot\ \  $\G, \bot\seq \D$

\vspace{0.3cm}
\ax{$\G, A \imp B \seq A$}
\ax{$\G, B \seq \D$}
\llab{\ilimp}
\binf{$\G, A \imp B \seq \D$}
\disp
\hfill
\ax{$\G, A\seq B$}
\llab{\irimp}
\uinf{$\G \seq A\imp B$}
\disp
\hfill
\ax{$\G \seq A$}
\ax{$\G \seq B$}
\llab{\irland}
\binf{$\G \seq A\land B$}
\disp

\vspace{0.2cm}
\ax{$\G, A, B \seq \D$}
\llab{\illand}
\uinf{$\G, A\land B \seq \D$}
\disp
\hfill
\ax{$\G \seq A_i$}
\llab{\irlor}
\rlab{($i \in\{1, 2\}$)}
\uinf{$\G \seq A_1\lor A_2$}
\disp
\hfill
\ax{$\G, A \seq \D$}
\ax{$\G, B \seq \D$}
\llab{\illor}
\binf{$\G, A \lor B \seq \D$}
\disp

\vspace{0.3cm}
\textbf{Modal rules} \hfill \quad

\vspace{0.2cm}
\ax{$\G, A \seq B$}
\llab{\ruleiMbox}
\uinf{$\G, \Box A \seq \Box B$}
\disp
\hfill
\ax{$\G, A \seq B$}
\llab{\ruleiMdiam}
\uinf{$\G, \diam A \seq \diam B$}
\disp
\hfill
\ax{$A, B \seq$}
\llab{\ruleidualandM}
\uinf{$\G, \Box A, \diam B \seq \D$}
\disp
\hfill
\ax{$\seq A$}
\llab{\ruleiNbox}
\uinf{$\G \seq \Box A$}
\disp

\vspace{0.2cm}
\ax{$\G, A \seq B$}
\llab{\ruleiCbox}
\uinf{$\G', \Box\G, \Box A \seq \Box B$}
\disp
\hfill
\ax{$\G, A \seq B$}
\llab{\ruleiCdiam}
\uinf{$\G', \Box\G, \diam A \seq \diam B$}
\disp
\hfill
\ax{$\G, A, B \seq$}
\llab{\ruleidualandC}
\uinf{$\G', \Box\G, \Box A, \diam B \seq \D$}
\disp

\vspace{0.2cm}
\ax{$A \seq$}
\llab{\ruleiNdiam}
\uinf{$\G, \diam A \seq \D$}
\disp
\hfill
\ax{$\G \seq A$}
\llab{\ruleiKbox}
\uinf{$\G', \Box\G \seq \Box A$}
\disp
\hfill
\ax{$\G, A \seq B$}
\llab{\ruleiKdiam}
\uinf{$\G', \Box\G, \diam A \seq \diam B$}
\disp
\hfill
\ax{$\G \seq A$}
\llab{\ruleiTdiam}
\uinf{$\G \seq \diam A$}
\disp

\vspace{0.2cm}
\ax{$\G, A \seq $}
\llab{\ruleidualandK}
\uinf{$\G', \Box\G, \diam A \seq \D$}
\disp
\hfill
\ax{$\G, \Box A, A \seq \D$}
\llab{\ruleiTbox}
\uinf{$\G, \Box A \seq \D$}
\disp
\hfill
\ax{$A \seq$}
\llab{\ruleiPbox}
\uinf{$\G, \Box A \seq \D$}
\disp
\hfill
\ax{$\seq A$}
\llab{\ruleiPdiam}
\uinf{$\G \seq\diam A$}
\disp

\vspace{0.2cm}
\ax{$A \seq B$}
\llab{\ruleiD}
\uinf{$\G, \Box A \seq \diam B$}
\disp
\hfill
\ax{$A, B \seq$}
\llab{\ruleiDbox}
\uinf{$\G, \Box A, \Box B \seq \D$}
\disp
\hfill
%\ax{$\G' \seq \D$}
%\llab{\ruleiCD}
%\uinf{$\G', \Box\G \seq  \diam\D$}
%\disp
\ax{$\G \seq A$}
\llab{\ruleiCD}
\uinf{$\G', \Box\G \seq  \diam A$}
\disp
\hfill
\ax{$\G \seq$}
\llab{\ruleiCDbox}
\uinf{$\G', \Box\G \seq \D$}
\disp
%
%\vspace{0.2cm}
%\ax{$\G, A \seq \D$}
%\llab{\ruleiTbox}
%\uinf{$\G, \Box A \seq \D$}
%\disp
%\hfill
%\ax{$\G \seq A$}
%\llab{\ruleiTdiam}
%\uinf{$\G \seq \diam A$}
%\disp
%\hfill
%\ax{$A \seq$}
%\llab{\ruleiPbox}
%\uinf{$\G, \Box A \seq \D$}
%\disp
%\hfill
%\ax{$\seq A$}
%\llab{\ruleiPdiam}
%\uinf{$\G \seq\diam A$}
%\disp
%
%\vspace{0.2cm}
%\ax{$A \seq B$}
%\llab{\ruleiD}
%\uinf{$\G, \Box A \seq \diam B$}
%\disp
%\hfill
%\ax{$A, B \seq$}
%\llab{\ruleiDbox}
%\uinf{$\G, \Box A, \Box B \seq \D$}
%\disp
%\hfill
%\ax{$\G' \seq \D$}
%\llab{\ruleiCD}
%\uinf{$\G, \Box\G' \seq  \diam\D$}
%\disp

\end{small}
\caption{\label{fig:wij sequent calculi} 
%Sequent calculus for classical $\K$
Sequent rules for
W-logics
% \wij-style constructive modal logics
%(where $|\G|\geq 0$, $|\D|\geq 0$, $i \in\{1, 2\}$).}
(where $|\G|, |\G'| \geq 0$, and  $0\leq |\D| \leq 1$).} %, and $i \in\{1, 2\}$).}
\end{figure}

The rules obtained by restricting sequents to at most one formula in the consequent
are displayed in Fig.~\ref{fig:wij sequent calculi}.
This restriction modifies %the classical calculi
the classical modal rules %essentiallly
 in two ways: %as follows.
first, the rules \ruledualorM, \ruledualorR,  and \ruleDdiam\
are dropped because they require at least two formulas in the consequent of sequents.
Second, %in the rules with a principal formula in the consequent of sequents, 
the right context is deleted from
all rules with a principal formula in the consequent of sequents,
namely \ruleTdiam, \ruleCbox, \ruleCdiam, and \ruleKbox.
Note in particular that $\diam\D$ is removed from \ruleCbox, \ruleCdiam, and \ruleKbox,
moreover $\diam A$ is removed from the premiss of \ruleiTdiam\
(the copy of $\diam A$ into the premiss of \ruleTdiam\ is needed in the classical calculus in order to ensure admissibility of 
right contraction, which is not expressible %here
in the \intu\ calculus).
Note also that \ruleKdiam\ and \ruleCD\ are split into two rules,
respectively \ruleiKdiam\ and \ruleidualandK, and \ruleiCD\ and \ruleiCDbox,
which correspond to the cases in which the consequent of the premiss of \ruleKdiam\ or \ruleCD\
is or is not empty.
Finally, note that %the rules
 \ruleiCbox\ and \ruleiKbox\ become % are
  equivalent.
%All other rules remain unchanged.
%As usual, 
Concerning the propositional rules, \limp\ is modified as usual by copying the principal implication into the left premiss in
order to ensure admissibility of contraction \cite{Troelstra:2000},
and \rlor\ is replaced by its single-succedent version. 
All other rules remain unchanged.
The resulting calculi $\seqWlogic$
are defined by
extending 
%as follows:
the set of \intu\ propositional %and structural 
rules with the following modal rules:
%
%%VERSIONE VECCHIA
%\begin{center}
%\begin{tabular}{ll}
%$\seqWM$ := \ruleiMbox\ + \ruleiMdiam\ + \ruleidualandM & 
%$\seqWMN$ := $\seqWM$ + \ruleiNbox\ + \ruleiNdiam \\ 
%$\seqWMC$ := \ruleiCbox\ + \ruleiCdiam\ + \ruleidualandC &
%$\seqWMP$ := $\seqWM$ + \ruleiPbox\ + \ruleiPdiam \\ 
%$\seqWK$ := \ruleiKbox\ + \ruleiKdiam &
%$\seqWMNP$ := $\seqWMN$ + \ruleiPbox\ + \ruleiPdiam \\
%$\seqWMCD$ := $\seqWMC$ + \ruleiCD &
%$\seqWMD$ := $\seqWM$ + \ruleiD\ + \ruleiDbox\ \tiz{+ \ruleiPdiam} \\
%$\seqWKD$ := $\seqWK$ + \ruleiCD & 
%$\seqWMND$ := $\seqWMN$ + \ruleiD\ + \ruleiDbox\ \tiz{+ \ruleiPdiam} \\
%\end{tabular}
%\end{center}
%Moreover, for every calculus $\seqWlogic$ above,
%$\Gone{WLT}$ := $\seqWlogic$ + \ruleiTbox\ + \ruleiTdiam.
%%FINE VERSIONE VECCHIA
\begin{center}
\begin{small}
\begin{tabular}{ll}
$\seqWM$ := \ruleiMbox\ + \ruleiMdiam\ + \ruleidualandM &
$\seqWMP$ := $\seqWM$ + \ruleiPbox\ + \ruleiPdiam \\
$\seqWMN$ := $\seqWM$ + \ruleiNbox\ + \ruleiNdiam &
$\seqWMNP$ := $\seqWMN$ + \ruleiPbox\ + \ruleiPdiam \\
$\seqWMC$ := \ruleiCbox\ + \ruleiCdiam\ + \ruleidualandC \\
\vspace{0.1cm}
$\seqWK$ := \ruleiKbox\ + \ruleiKdiam\ + \ruleidualandK \\
$\seqWMD$ := $\seqWM$ + \ruleiD\ + \ruleiDbox\ + \ruleiPbox\ + \ruleiPdiam &
$\seqWMT$ := $\seqWM$ + \ruleiTbox\ + \ruleiTdiam \\
$\seqWMND$ := $\seqWMN$ + \ruleiD\ + \ruleiDbox\ + \ruleiPbox\ + \ruleiPdiam &
$\seqWMNT$ := $\seqWMN$ + \ruleiTbox\ + \ruleiTdiam \\
$\seqWMCD$ := $\seqWMC$ + \ruleiCD\ + \ruleiCDbox &
$\seqWMCT$ := $\seqWMC$ + \ruleiTbox\ + \ruleiTdiam \\
$\seqWKD$ := $\seqWK$ + \ruleiCD\ + \ruleiCDbox &
$\seqWKT$ := $\seqWK$ + \ruleiTbox\ + \ruleiTdiam \\
\end{tabular}
\end{small}
\end{center}

%\tiz{NB: forse \ruleiPdiam\ in calcoli per D serve}.

Note  that the modal rules of $\seqWK$
coincide with those of \wij~\cite{wij}
(except that they have side context in the conclusion in order to embed weakening
in their application).
$\seqWK$ coincides with the calculus $\mathsf{G.CCDL^p}$
for $\WK$ proposed in \cite{dal:JPL}.

%The \const\ calculi are defined essentially by replacing the classical rules with their \intu\ versions, 
%and 
%removing \rctr, \ruledualorM, \ruledualorC,  and \ruleDdiam.
%Note however that 
%%The only exeptions are $\seqWMD$ and $\seqWMND$ where
%\ruleDdiam\ is replaced with \ruleiPdiam\
%%, since
%because \rulePdiam\ is derivable from \ruleDdiam\ (cf.~Sec.~\ref{sec:prel})
%and it is allowed in the \const\ calculi.
%%requiring only one formula in the right
%Note also that the modal rules of $\seqWK$
%coincide with those of \wij~\cite{wij}.

From the point of view of the derivable principles, we observe two main
consequences of %this restriction.
the restriction of the calculi to single-succedent sequents.
First, 
%beause of the elimination of \ruledualorM\ and \ruledualorC,
%and the restriction of \ruleKbox, 
%$\Box A \lor  \diam \neg A$ 
%\tiz{the rule of disjunctive duality $A \lor B / \Box A \lor \diam B$}
the rule \Rdualor\
is no longer derivable
in the calculi.
This is due to the %elimination
absence of \ruledualorM\ and \ruledualorC,
and the %restriction of 
elimination of $\diam$-formulas from the conclusion of
\ruleiKbox.
Second, \axCdiam\ is not derivable in $\seqWMC$, $\seqWK$
and their extensions,
this is due to the restriction of \ruleiCdiam\ and \ruleiKdiam\ to only one $\diam$-formula
in the right-hand side of the conclusion.
By contrast, all other modal principles from Fig.~\ref{fig:axioms}
are still derivable in the corresponding calculi
%derivations are displayed in Fig.~\ref{fig:derivations}.
(cf.~derivations in Fig.~\ref{fig:derivations}).

\begin{figure}
\centering
\begin{small}
\begin{tabular}{lll}
\vspace{0.3cm}
\ax{$A, A \imp \bot \seq A$}
\ax{$A, \bot \seq$}
\rlab{\ilimp}
\binf{$A, A \imp \bot \seq$}
\rlab{\ruleidualandM}
\uinf{$\Box A, \diam (A \imp \bot) \seq\bot$}
\rlab{\illand}
\uinf{$\Box A \land \diam (A \imp \bot) \seq\bot$}
\rlab{\irimp}
\uinf{$\seq \Box A \land \diam (A \imp \bot) \imp\bot$}
\disp
&&
\ax{$A, B \seq A$}
\ax{$A, B \seq B$}
\rlab{\irland}
\binf{$A, B \seq A \land B$}
\rlab{\ruleiCbox}
\uinf{$\Box A, \Box B \seq \Box(A \land B)$}
\rlab{\illand}
\uinf{$\Box A \land \Box B \seq \Box(A \land B)$}
\rlab{\irimp}
\uinf{$\seq \Box A \land \Box B \imp \Box(A \land B)$}
\disp
\\

\ax{$A \imp B, A \seq A$}
\ax{$B, A \seq B$}
\rlab{\ilimp}
\binf{$A \imp B, A \seq B$}
\rlab{\ruleiCbox}
\uinf{$\Box(A \imp B), \Box A  \seq \Box B$}
\rlab{\irimp}
\uinf{$\Box(A \imp B) \seq \Box A \imp \Box B$}
\rlab{\irimp}
\uinf{$\seq \Box(A \imp B) \imp (\Box A \imp \Box B)$}
\disp
&&
\ax{$A \imp B, A \seq A$}
\ax{$B, A \seq B$}
\rlab{\ilimp}
\binf{$A \imp B, A \seq B$}
\rlab{\ruleiCdiam}
\uinf{$\Box(A \imp B), \diam A  \seq \diam B$}
\rlab{\irimp}
\uinf{$\Box(A \imp B) \seq \diam A \imp \diam B$}
\rlab{\irimp}
\uinf{$\seq \Box(A \imp B) \imp (\diam A \imp \diam B)$}
\disp
\end{tabular}

\vspace{0.3cm}
\ax{$\bot\seq\bot$}
\rlab{\irimp}
\uinf{$\seq \bot\imp\bot$}
\rlab{\ruleiNbox}
\uinf{$\seq\Box(\bot\imp\bot)$}
\disp
\ \ 
\ax{$\bot \seq$}
\rlab{\ruleiNdiam}
\uinf{$\diam\bot\seq\bot$}
\rlab{\irimp}
\uinf{$\seq\diam\bot\imp\bot$}
\disp
\ \
\ax{$\Box A, A \seq A$}
\rlab{\ruleiTbox}
\uinf{$\Box A \seq A$}
\rlab{\irimp}
\uinf{$\seq \Box A \imp A$}
\disp
\ \
\ax{$A \seq A$}
\rlab{\ruleiTdiam}
\uinf{$A \seq \diam A$}
\rlab{\irimp}
\uinf{$\seq A \imp \diam A$}
\disp

\vspace{0.3cm}
\ax{$A \seq A$}
\rlab{\ruleiD}
\uinf{$\Box A \seq \diam A$}
\rlab{\irimp}
\uinf{$\seq \Box A \imp \diam A$}
\disp
\quad
\ax{$\bot \seq$}
\rlab{\ruleiPbox}
\uinf{$\Box\bot\seq\bot$}
\rlab{\irimp}
\uinf{$\seq\Box\bot\imp\bot$}
\disp
\quad
\ax{$\bot\seq\bot$}
\rlab{\irimp}
\uinf{$\seq \bot\imp\bot$}
\rlab{\ruleiPdiam}
\uinf{$\seq\diam(\bot\imp\bot)$}
\disp
\end{small}
\caption{\label{fig:derivations} Derivations of the modal axioms.}
\end{figure}

In the following, we denote with $\seqWstar$ any constructive calculus defined above.
%We now prove that the calculi $\seqWstar$ enjoy admissibility of strucural rules and cut,
%then we present equivalent axiomatic systems. % equivalent to the calculi.
As usual, we say that a rule 
is \emph{admissible} in $\seqWstar$ if whenever the premisses are derivable, the conclusion is
also derivable, and that
a single-premiss rule
is \emph{height-preserving admissible}
%(hp-admissible for short) 
if whenever the premiss is derivable, %then 
the conclusion is derivable with
a derivation of at most the same height.
We now prove that the calculi $\seqWstar$ enjoy admissibility of structural rules and cut,
then we present equivalent axiomatic systems. % equivalent to the calculi.

%\begin{theorem}[Admissibility of structural rules and cut]\label{th:cut elim}
%The following rules \ilwk, \irwk\ and \ilctr\ are height-preserving admissible in $\seqWstar$,
%moreover the following rule \cut\ is admissible in $\seqWstar$:
%\begin{center}
%\begin{small}
%\ax{$\G \seq \D$}
%\llab{\ilwk}
%\uinf{$\G, A \seq \D$}
%\disp
%\hfill
%\ax{$\G \seq$}
%\llab{\irwk}
%\uinf{$\G \seq A$}
%\disp
%\hfill
%\ax{$\G, A, A\seq \D$}
%\llab{\ilctr}
%\uinf{$\G, A \seq \D$}
%\disp
%\hfill
%\ax{$\G \seq A$}
%\ax{$\G', A \seq \D$}
%\llab{\cut}
%\binf{$\G, \G' \seq \D$}
%\disp
%\end{small}
%\end{center}
%\end{theorem}

\begin{proposition}[Admissibility of structural rules]\label{prop:adm str rules}
The following rules are height-preserving admissible in $\seqWstar$:
\begin{center}
%\begin{small}
\ax{$\G \seq \D$}
\llab{\ilwk}
\uinf{$\G, A \seq \D$}
\disp
\qquad
\ax{$\G \seq$}
\llab{\irwk}
\uinf{$\G \seq A$}
\disp
\qquad
\ax{$\G, A, A\seq \D$}
\llab{\ilctr}
\uinf{$\G, A \seq \D$}
\disp.
%\end{small}
\end{center}
\end{proposition}
\begin{proof}
Height-preserving admissibility of \ilwk, \irwk, and \ilctr\
is proved by induction on the height of the derivation of their premiss,
taking into account the last rule applied in the derivation.
For \ilwk\ and \irwk\ the proof is straightforward, we consider
some examples for \ilctr\ involving the modal rules.
The derivations on the left are converted into the derivations on the right,
%including
which include applications of \ilctr\ which are height-preserving admissible by \ih:
\begin{center}
\begin{small}
\begin{tabular}{ccc}
\vspace{0.2cm}
\ax{$A, A \seq$}
\llab{\ruleiDbox}
\uinf{$\G, \Box A, \Box A \seq \D$}
%\llab{\ilctr}
%\uinf{$\G, \Box A \seq \D$}
\disp
& $\leadsto$ &
\ax{$A, A \seq$}
\rlab{\ilctr}
\uinf{$A \seq$}
\rlab{\ruleiPbox}
\uinf{$\G, \Box A \seq \D$}
\disp
\\

\vspace{0.2cm}
\ax{$\G, \Box A, \Box A, A \seq \D$}
\llab{\ruleiTbox}
\uinf{$\G, \Box A, \Box A \seq \D$}
%\llab{\ilctr}
%\uinf{$\G, \Box A \seq \D$}
\disp
& $\leadsto$ &
\ax{$\G, \Box A, \Box A, A \seq \D$}
\rlab{\ilctr}
\uinf{$\G, \Box A, A \seq \D$}
\rlab{\ruleiTbox}
\uinf{$\G, \Box A \seq \D$}
\disp
\\

\vspace{0.2cm}
\ax{$\G, A, A, B \seq C$}
\llab{\ruleiCdiam}
\uinf{$\G', \Box\G, \Box A, \Box A, \diam B \seq \diam C$}
\disp
& $\leadsto$ &
\ax{$\G, A, A, B \seq C$}
\rlab{\ilctr}
\uinf{$\G, A, B \seq C$}
\rlab{\ruleiCdiam}
\uinf{$\G', \Box\G, \Box A, \diam B \seq \diam C$}
\disp
\\

\ax{$\G, A, B \seq C$}
\llab{\ruleiCdiam}
\uinf{$\G', \Box A, \Box\G, \Box A, \diam B \seq \diam C$}
\disp
& $\leadsto$ &
\ax{$\G, A, B \seq C$}
\rlab{\ruleiCdiam}
\uinf{$\G', \Box\G, \Box A, \diam B \seq \diam C$}
\disp
\end{tabular}
\end{small}
\end{center}

In the last example, 
%the first application of \ruleiCdiam\ introduces one occurrence of $\Box A$ as part of the side context,
one occurrence of $\Box A$ is introduced by the first application of \ruleiCdiam\ as part of the side context,
while this is not the case in the second application of \ruleiCdiam.
\end{proof}

\begin{theorem}[Cut admissibility]\label{th:cut elim}
The following cut rule is admissible in  $\seqWstar$:
\begin{center}
\ax{$\G \seq A$}
\ax{$\G', A \seq \D$}
\llab{\cut}
\binf{$\G, \G' \seq \D$}
\disp.
\end{center}
\end{theorem}
\begin{proof}
%Admissibility of \cut\ is proved by 
By induction on lexicographically ordered pairs ($c$, $h$),
where $c$ is the complexity of the cut formula
(\ie, the number of binary
connectives or modalities occurring in it),
and $h = h_1+h_2$, called cut height,
is the sum of the heights of the derivations of the premisses of \cut. 
As usual, 
we distinguish some cases according to whether the cut formula is or not principal
in the last rules applied in the derivation of the premisses of \cut. 
For the cases where the last rules applied in the derivation of the premisses of \cut\
are propositional we refer to \cite[Ch.~4]{Troelstra:2000}.
Here we only show a few most relevant cases
involving modal rules, the other cases are similar.

\begin{itemize}
\item[(i)]
%(i) 
The cut formula is not principal in the last rule application in the derivation of the left premiss of \cut.
We consider the following two examples,
where the derivation on the left is converted into the derivation on the right:
\begin{center}
\begin{small}
\ax{$A, B \seq$}
\llab{\ruleidualandM}
\uinf{$\G, \Box A, \diam B \seq C$}
\ax{$\G', C \seq \D$}
\llab{\cut}
\binf{$\G, \G', \Box A, \diam B \seq \D$}
\disp
\ \ $\leadsto$ \ \
\ax{$A, B \seq$}
\rlab{\ruleidualandM}
\uinf{$\G, \G', \Box A, \diam B \seq \D$}
\disp

\vspace{0.2cm}
\ax{$\G, \Box A, A \seq B$}
\llab{\ruleiTbox}
\uinf{$\G, \Box A \seq B$}
\ax{$\G', B \seq \D$}
\llab{\cut}
\binf{$\G, \G', \Box A \seq \D$}
\disp
\ \ $\leadsto$ \ \
\ax{$\G, \Box A, A \seq B$}
\ax{$\G', B \seq \D$}
\rlab{\cut}
\binf{$\G, \G', \Box A, A \seq \D$}
\rlab{\ruleiTbox}
\uinf{$\G, \G', \Box A \seq \D$}
\disp
\end{small}
\end{center}

%\vspace{0.1cm}
%(ii) 
\item[(ii)]
The cut formula is not principal in the last rule application in the derivation of the right premiss of \cut.
We consider the following example:
\begin{center}
\begin{small}
\ax{$\G \seq A$}
\ax{$\G'', B \seq C$}
\rlab{\ruleiCdiam}
\uinf{$\G', A, \Box \G'', \diam B \seq \diam C$}
\llab{\cut}
\binf{$\G, \G', \Box \G'', \diam B \seq \diam C$}
\disp
\ \ $\leadsto$ \ \
\ax{$\G'', B \seq C$}
\rlab{\ruleiCdiam}
\uinf{$\G, \G', \Box \G'', \diam B \seq \diam C$}
\disp
\end{small}
\end{center}

%\vspace{0.1cm}
%(iii) 
\item[(iii)]
The cut formula is principal in the last rule application in the derivations of both premisses of \cut.
We consider the following three examples,
where $R^*$ denotes multiple applications of the rule $R$:

%\vspace{0.1cm}
\begin{small}
\begin{center}
(\ruleiCdiam; \ruleidualandC) \hfill \quad

\ax{$\G, A \seq B$}
\llab{\ruleiCdiam}
\uinf{$\G', \Box\G, \diam A \seq \diam B$}
\ax{$\G'', C, B \seq$}
\rlab{\ruleidualandC}
\uinf{$\G''', \Box\G'', \Box C, \diam B \seq \D$}
\rlab{\cut}
\binf{$\G', \G''', \Box\G, \Box\G'', \Box C, \diam A \seq \D$}
\disp

\vspace{0.15cm}
$\vleadsto$

\vspace{0.15cm}
\ax{$\G, A \seq B$}
\ax{$\G'', C, B \seq$}
\rlab{\cut}
\binf{$\G, \G'', C, A \seq$}
\rlab{\ruleidualandC}
\uinf{$\G', \G''', \Box\G, \Box\G'', \Box C, \diam A \seq \D$}
\disp

\vspace{0.2cm}
(\ruleiTdiam; \ruleiCdiam) \hfill \quad

\ax{$\G \seq A$}
\llab{\ruleiTdiam}
\uinf{$\G \seq \diam A$}
\ax{$\G', A \seq B$}
\rlab{\ruleiCdiam}
\uinf{$\G'', \Box \G', \diam A \seq \diam B$}
\rlab{\cut}
\binf{$\G, \G'', \Box \G' \seq \diam B$}
\disp
\ \ $\leadsto$ \ \
\ax{$\G \seq A$}
\ax{$\G', A \seq B$}
\rlab{\cut}
\binf{$\G, \G' \seq B$}
\rlab{\ruleiTdiam}
\uinf{$\G, \G' \seq \diam B$}
\rlab{\ruleiTbox$^*$}
\uinf{$\G, \Box\G' \seq \diam B$}
\rlab{\ilwk$^*$}
\uinf{$\G, \G'', \Box\G' \seq \diam B$}
\disp

%\vspace{0.2cm}
%\newpage
(\ruleiNbox; \ruleiD) \hfill \quad

\ax{$\seq A$}
\llab{\ruleiNbox}
\uinf{$\G \seq \Box A$}
\ax{$A \seq B$}
\rlab{\ruleiD}
\uinf{$\G', \Box A \seq \diam B$}
\rlab{\cut}
\binf{$\G, \G' \seq \diam B$}
\disp
\ \ $\leadsto$ \ \
\ax{$\seq A$}
\ax{$A \seq B$}
\rlab{\cut}
\binf{$\seq B$}
\rlab{\ruleiPdiam}
\uinf{$\G, \G' \seq \diam B$}
\disp
\end{center}
\end{small}
\end{itemize}
\end{proof}

\subsection{Axiom systems}
For each \const\ calculus $\seqWL$, we now define an equivalent axiomatic system. % $\WLogic$.
The logics $\WLL$ are defined in the language $\lan$ extending 
(any axiomatisation of) $\IPL$
with the following modal axioms and rules from Fig.~\ref{fig:axioms}:
%% buono con tutti i sistemi
\begin{center}
\begin{tabular}{lllll}
$\WM$ := \axdualand\ + \monbox\ + \mondiam &&  $\WMND$ := $\WMN$ + \axD  \\
$\WMN$ := $\WM$ + \axNbox && $\WMCD$ := $\WMC$ + \axD\ + \axPdiam \\
$\WMC$ := $\WM$ + \axCbox\ + \axKdiam && $\WKD$ := $\WK$ + \axD   \\
$\WK$ := $\WMC$ + \axNbox &&  $\WMT$ := $\WM$ + \axTbox\ + \axTdiam \\
$\WMP$ := $\WM$ + \axPdiam &&  $\WMNT$ := $\WMN$ + \axTbox\ + \axTdiam \\  
$\WMNP$ := $\WMN$ + \axPdiam &&  $\WMCT$ := $\WMC$ + \axTbox\ + \axTdiam \\ 
$\WMD$ := $\WM$ + \axD\ + \axPdiam && $\WKT$ := $\WK$ + \axTbox\ + \axTdiam \\
\end{tabular}
\end{center}

%%BUONA CON FRECCE CURVE
%\begin{figure}
%\centering
%\begin{small}
%\begin{tikzpicture}
%    \node (M) at  (0,0)  {$\WM$};
%    \node (MN) at (1, 0.9) {$\WMN$};
%    \node  (MC) at (1, -0.9) {$\WMC$};
%    \node (K) at (2, 0) {$\WK$};
%
%    \node (MP) at  (3.4,0)  {$\WMP$};
%    \node (MNP) at (4.4, 0.9) {$\WMNP$};
%
%    \node (MD) at  (5.5,0)  {$\WMD$};
%    \node (MND) at (6.5, 0.9) {$\WMND$};
%    \node  (MCD) at (6.5, -0.9) {$\WMCD$};
%    \node (KD) at (7.5, 0) {$\WKD$};
%
%    \node (MT) at  (8.8,0)  {$\WMT$};
%    \node (MNT) at (9.8, 0.9) {$\WMNT$};
%    \node  (MCT) at (9.8, -0.9) {$\WMCT$};
%    \node (KT) at (10.8, 0) {$\WKT$};
%
%	\draw[->] (M) -- (MN);
%	\draw[->] (M) -- (MC);
%	\draw[->] (MN) -- (K);	
%	\draw[->] (MC) -- (K);
%	\draw[->] (MP) -- (MNP);
%	\draw[->] (MD) -- (MND);
%	\draw[->] (MD) -- (MCD);
%	\draw[->] (MND) -- (KD);	
%	\draw[->] (MCD) -- (KD);
%	\draw[->] (MT) -- (MNT);
%	\draw[->] (MT) -- (MCT);
%	\draw[->] (MNT) -- (KT);	
%	\draw[->] (MCT) -- (KT);
%
%	\draw[->] (MN) -- (MNP);
%	\draw[->] (MNP) -- (MND);
%	\draw[->] (MND) -- (MNT);
%	\draw[->] (MC) -- (MCD);
%	\draw[->] (MCD) -- (MCT);
%	\draw[->] (MP) -- (MD);
%
%	\path[->] (M) edge [bend left=25] (MP);
%	\path[->] (MD) edge [bend left=25] (MT);
%	\path[->] (KD) edge [bend right=25] (KT);
%	\path[->] (K) edge [bend right=18] (KD);
%\end{tikzpicture}
%\end{small}
%\caption{\label{fig:dyag const} Dyagram of constructive modal logics.}
%\end{figure}

\begin{figure}
\centering
%\begin{small}
\begin{footnotesize}
\begin{tikzpicture}
%%GRANDE BUONO
%    \node (M) at  (0,0)  {$\WM$};
%    \node (MN) at (0.6, 1.2) {$\WMN$};
%    \node  (MC) at (1.2, -0.6) {$\WMC$};
%    \node (K) at (1.8, 0.6) {$\WK$};
%
%    \node (MP) at  (3.3,0)  {$\WMP$};
%    \node (MNP) at (3.9, 1.2) {$\WMNP$};
%
%    \node (MD) at  (5.4,0)  {$\WMD$};
%    \node (MND) at (6, 1.2) {$\WMND$};
%    \node  (MCD) at (6.6, -0.6) {$\WMCD$};
%    \node (KD) at (7.2, 0.6) {$\WKD$};
%
%    \node (MT) at  (8.6,0)  {$\WMT$};
%    \node (MNT) at (9.2, 1.2) {$\WMNT$};
%    \node  (MCT) at (9.8, -0.6) {$\WMCT$};
%    \node (KT) at (10.4, 0.6) {$\WKT$};
%%MEDIO BUONO
%    \node (M) at  (0,0)  {$\WM$};
%    \node (MN) at (0.5, 1) {$\WMN$};
%    \node  (MC) at (1, -0.5) {$\WMC$};
%    \node (K) at (1.5, 0.5) {$\WK$};
%
%    \node (MP) at  (3.3,0)  {$\WMP$};
%    \node (MNP) at (3.8, 1) {$\WMNP$};
%
%    \node (MD) at  (5.4,0)  {$\WMD$};
%    \node (MND) at (5.9, 1) {$\WMND$};
%    \node  (MCD) at (6.4, -0.5) {$\WMCD$};
%    \node (KD) at (6.9, 0.5) {$\WKD$};
%
%    \node (MT) at  (8.6,0)  {$\WMT$};
%    \node (MNT) at (9.1, 1) {$\WMNT$};
%    \node  (MCT) at (9.6, -0.5) {$\WMCT$};
%    \node (KT) at (10.1, 0.5) {$\WKT$};
    \node (M) at  (0,0)  {$\WM$};
    \node (MN) at (0.5, 1) {$\WMN$};
    \node  (MC) at (1, -0.5) {$\WMC$};
    \node (K) at (1.5, 0.5) {$\WK$};

    \node (MP) at  (3.,0)  {$\WMP$};
    \node (MNP) at (3.5, 1) {$\WMNP$};

    \node (MD) at  (5.2,0)  {$\WMD$};
    \node (MND) at (5.7, 1) {$\WMND$};
    \node  (MCD) at (6.2, -0.5) {$\WMCD$};
    \node (KD) at (6.7, 0.5) {$\WKD$};

    \node (MT) at  (8.3,0)  {$\WMT$};
    \node (MNT) at (8.8, 1) {$\WMNT$};
    \node  (MCT) at (9.3, -0.5) {$\WMCT$};
    \node (KT) at (9.8, 0.5) {$\WKT$};

	\draw[->] (M) -- (MN);
	\draw[->] (M) -- (MC);
	\draw[->, dashed] (MN) -- (K);	
	\draw[->, dashed] (MC) -- (K);
	\draw[->] (MP) -- (MNP);
	\draw[->] (MD) -- (MND);
	\draw[->] (MD) -- (MCD);
	\draw[->, dashed] (MND) -- (KD);	
	\draw[->, dashed] (MCD) -- (KD);
	\draw[->] (MT) -- (MNT);
	\draw[->] (MT) -- (MCT);
	\draw[->] (MNT) -- (KT);	
	\draw[->] (MCT) -- (KT);

	\draw[->] (MN) -- (MNP);
	\draw[->] (MNP) -- (MND);
	\draw[->] (MND) -- (MNT);
	\draw[->] (MC) -- (MCD);
	\draw[->] (MCD) -- (MCT);
	\draw[->] (MP) -- (MD);

	\draw[->] (M) -- (MP);
	\draw[->] (MD) -- (MT);
	\draw[->, dashed] (KD) -- (KT);
	\draw[->, dashed] (K) -- (KD);
\end{tikzpicture}
%\end{small}
\end{footnotesize}
\caption{\label{fig:dyag const} Dyagram of constructive modal logics.}
\end{figure}
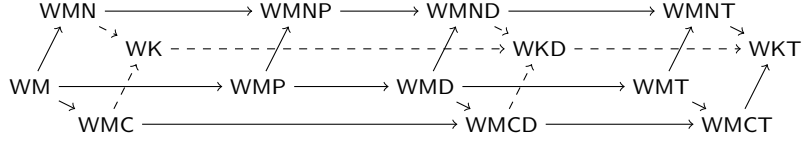

In the following, we will refer to these systems as \emph{W-logics}.
%Frase vecchia
%Moreover, we denote $\WL$ any W-logic, $\WLC$ any W-logic with axioms \axCbox\ and \axKdiam,
%and $\WLT$ any W-logic with axioms \axTbox\ and \axTdiam.
Moreover, we denote $\WL$ any W-logic, and we denote $\WCstar$, resp.~$\WDstar$, resp.~$\WTstar$
any W-logic with axioms \axCbox\ and \axKdiam,
resp.~with axiom \axD, resp.~with axioms \axTbox\ and \axTdiam.
As usual, 
%given a logic $\WL$, a formula $A$, and a set of formulas $\Sigma$,
we say that $A$ is a theorem of $\WL$, written $\WL\vd A$, if there is a finite sequence of formulas 
ending with $A$ in which every formula is an axiom of $\WL$, or it is obtained from previous formulas by 
%means
the application 
of a rule
of $\WL$.
Moreover,
%Given a logic $\WL$ and a set of formulas $\Sigma$, 
%we say that $A$ is deducible from $\Sigma$ in $\WL$, 
we say that $A$ is deducible in $\WL$ from a set of formulas $\Sigma$, 
written
$\Sigma\vd_{\WL} A$, if there is a finite set $\{B_1, ..., B_n\}\subseteq\Sigma$ such that 
$\vd_{\WL} B_1\land ... \land B_n \imp A$.
Furthermore, given two axiomatic systems $\mathsf{L_1}$ and  $\mathsf{L_2}$,
we say that  $\mathsf{L_1}$ is included in  $\mathsf{L_2}$ if  $\mathsf{L_1}\vd A$ entails
 $\mathsf{L_2}\vd A$ for all $A\in\lan$,
 and that $\mathsf{L_1}$ and  $\mathsf{L_2}$ are equivalent if they derive exactly the same theorems.

Concerning W-logics specifically, note
%
%Note 
that \axCdiam\ is not an axiom of $\WMC$ because it is not derivable in $\seqWMC$,
\axCdiam\ must be replaced by \axKdiam\ which is instead derivable in the calculus.
Note also that \axPdiam\ must be included in the axiomatisation of $\WMD$ and $\WMCD$
as it is not derivable from \axD\ in the intuitionistic systems.
The relations among the W-logics are displayed in Fig.~\ref{fig:dyag const}.

We prove some basic results about W-logics.

\begin{proposition}\label{prop:basic res w}
$(i)$ $\WM \vd \textup{\Rdualand}$.
$(ii)$ $\WMN \vd \textup{\nec}$.
$(iii)$ $\WMN \vd \textup{\axNdiam}$.
$(iv)$ $\WMC \vd \textup{\axKbox}$.
$(v)$ $\WMP \vd \textup{\axPbox}$.
$(vi)$ $\WMND \vd \textup{\axPdiam}$.
\end{proposition}
\begin{proof}
(i) From $\neg(A \land B)$, by $\IPL$ we obtain $B \imp \neg A$, then by \mondiam, $\diam B \imp \diam \neg A$,
thus $\neg \diam \neg A \imp \neg \diam B$.
Moreover from \axdualand\ we have $\Box A \imp \neg \diam \neg A$, 
thus $\Box A \imp \neg \diam B$, therefore $\neg(\Box A \land \diam B)$.
(ii) From $A$, by $\IPL$ we have $\top\imp A$, then by \monbox, $\Box\top\imp\Box A$,
then by \axNbox, $\Box A$.
%
%(iii) By \axdualand, $\Box\top\imp\neg\diam\neg\top$,
%then by \axNbox, $\neg\diam\neg\top$, that is $\neg\diam\bot$.
(iii) By \Rdualand, $\Box\top\imp\neg\diam\bot$,
then by \axNbox,  $\neg\diam\bot$.
(iv) By \axCbox, $\Box(A \imp B) \land \Box A \imp \Box((A\imp B) \land A)$,
and by \monbox, $\Box((A\imp B) \land A) \imp \Box B$,
thus by $\IPL$,  $\Box(A \imp B) \imp (\Box A \imp \Box B)$.
(v) By \axdualand, $\diam\top\imp \neg\Box\bot$, then by \axPdiam, $\neg\Box\bot$.
(vi) By \axD, $\Box\top\imp \diam\top$, then by \axNbox, $\diam\top$.
\end{proof}

We recall that \wij's original axiomatisation of $\WK$ \cite{wij} was given by
$\IPL + \textup{\nec} + \textup{\axKbox} +  \textup{\axKdiam} + \textup{\axNdiam}$.
It is easy to verify that \axdualand, \monbox, \mondiam, \axCbox, and \axNbox\ are all derivable
in \wij's original system. 
Then from Proposition~\ref{prop:basic res w}
it follows in particular that the axiomatisation of $\WK$ 
considered here
is equivalent to \wij's one.

We now prove that the systems $\WL$ are equivalent to the corresponding calculi.

%From the above proposition,
%it follows in particular that the considered axiomatisation of $\WK$ is equivalent to the one by \wij~\cite{wij}.
%We now prove that the systems $\WL$ are equivalent to the corresponding calculi.

\begin{theorem}\label{th:eq seq ax}
%$\WL \vd A$ if and only if $\seqWL \vd \ \seq A$.\nb{o versione con formula int?}
$\seqWstar \vd \G \seq \D$ if and only if $\WL \vd \AND\G \imp \OR\D$. %\nb{adeguare notaz. seqWL e Wstar}
\end{theorem}
\begin{proof}
%From left to right,
From right to left, it is easy to see that all modal axioms are derivable in the corresponding calculi
(cf.~derivations in Fig.~\ref{fig:derivations}),
observing that initial sequents \init\ can be generalised as usual to arbitrary formulas $A$.
For \monbox, from $\seq A \imp B$ and $A \imp B, A \seq B$, by \cut\ 
(which has been proved admissible)
we obtain $A \seq B$,
then by \ruleiMbox, $\Box A \seq \Box B$, and by \irimp, $\seq\Box A \imp \Box B$.
\mondiam\ is derived similarly.
The derivations of the intuitionistic axioms are standard, moreover modus ponens is simulated by \cut\
in the usual way.

For the other direction, 
we can consider standard derivations of the propositional %and structural
 rules.
Here we show that %every rule \ruleiX\ 
for every modal sequent rule %\ruleiX\ 
of $\seqWstar$
%\ax{$\G' \seq \D'$}
%\ax{($\G'' \seq \D''$)}
%\llab{\ruleiX}
%\binf{$\G \seq \D$}
%\disp
with premiss $\G \seq \D$ and conclusion $\G'\seq\D'$,
the %corresponding %Hilbert
Hilbert-style
rule
$\AND\G \imp \OR\D$ / $\AND\G' \imp \OR\D'$
is derivable 
in the corresponding system $\WL$. %\nb{adeguare notaz. seqWL e Wstar}
%using the corresponding axiom or rule \axX. 
%
We only consider some relevant examples, 
the other derivations are similar.

%(\ruleiCbox) 
\begin{enumerate}[leftmargin=*, align=left]
\item[(\ruleiCbox)]
From $\AND\G\land A \imp B$, by \monbox\ we get $\Box(\AND\G\land A) \imp \Box B$,
moreover by \axCbox, $\AND\Box\G \land \Box A \imp \Box(\AND\G\land A)$,
then $\AND\G' \land \AND\Box\G \land \Box A \imp \Box B$.

\item[(\ruleiCdiam)]
From $\AND\G\land A \imp B$, we get $\AND\G \imp (A \imp B)$, then by \monbox,
$\Box\AND\G \imp \Box(A \imp B)$.
By \axCbox\ we have $\AND\Box\G \imp \Box(A \imp B)$, then by \axKdiam,
$\AND\Box\G \imp (\diam A \imp \diam B)$,
thus $\AND\G' \land \AND\Box\G \land \diam A \imp \diam B$.

\item[(\ruleidualandC)]
From $\AND\G\land A \land B \imp \bot$, we get $\AND\G\land A \imp \neg B$, 
then by \monbox, $\Box (\AND\G\land A) \imp \Box \neg B$,
and by \axCbox, $\AND\Box\G\land \Box A \imp \Box \neg B$.
Moreover by \axdualand, $\Box\neg B \imp \neg \diam B$,
thus $\AND\G' \land \AND\Box\G\land \Box A \land \diam B \imp C$ for any $C$.

%(\ruleiCD)
%If $\D = B$ and $\G\not=\emptyset$, then from $\AND\G\imp B$, by \monbox\ we get $\Box\AND\G\imp \Box B$,
%then by \axCbox, $\AND\Box\G\imp \Box B$, and by \axD, $\AND\Box\G\imp \diam B$.
%If $\D = B$ and $\G=\emptyset$, then from $B$ we get $\top\imp B$, then 
%by \mondiam, $\diam\top\imp\diam B$, and by \axPdiam, $\diam B$.
%Finally, if $\D=\emptyset$, then from $\AND\G\imp \bot$ we get $\AND\G'\imp\neg B$ for some
%$B\in\G$ and $\G'=\G\setminus B$.
%Then by \mondiam, $\diam \AND\G'\imp \diam\neg B$,
%thus by \axD, $\Box \AND\G'\imp\diam\neg B$,
%and by \axCbox, $\AND\Box\G'\imp\diam\neg B$, 
%moreover by \axdualand, $\AND\Box\G'\imp\neg\Box B$,
%therefore $\AND\Box\G'\land\Box B \imp \bot$.

\item[(\ruleiCD)]
If $\G\not=\emptyset$, then from $\AND\G\imp B$, by \monbox\ we get $\Box\AND\G\imp \Box B$,
then by \axCbox, $\AND\Box\G\imp \Box B$, and by \axD, $\AND\Box\G\imp \diam B$,
thus $\AND\G' \land \AND\Box\G\imp \diam B$.
If $\G=\emptyset$, then from $B$ we get $\top\imp B$, then 
by \mondiam, $\diam\top\imp\diam B$, and by \axPdiam, $\diam B$,
thus $\AND\G' \imp \diam B$.

\item[(\ruleiCDbox)]
From $\AND\G\imp \bot$ we get $\AND\G''\imp\neg B$ for some
$B\in\G$ and $\G''=\G\setminus B$.
Then by \mondiam, $\diam \AND\G''\imp \diam\neg B$,
thus by \axD, $\Box \AND\G''\imp\diam\neg B$,
by \axCbox, $\AND\Box\G''\imp\diam\neg B$, 
and by \axdualand, $\AND\Box\G''\imp\neg\Box B$,
thus $\AND\G' \land \AND\Box\G \imp C$ for any $C$.
\end{enumerate}
\end{proof}

\subsection{Some properties of W-logics}

Cut-free sequent calculi are a very powerful tool for the analysis of logical systems.
In this subsection, we present some fundamental properties of W-logics
that % trivially
%immediately
easily follow from those of their sequent calculi.

%Finally, it is also worth establishing how classical and \wij-style modal logics are related
%from the point of view of the axiomatic systems.
%We conclude this section with the following observation.

We start considering the following result, which establishes
how classical and \wij-style modal logics are related
from the point of view of the axiomatic systems.

\begin{theorem}
Let $\logic$ be any classical modal logic from Sec.~\ref{sec:prel},
and $\WLL$ be the corresponding W-logic.
Then $\logic$ is equivalent to $\WLL + A \lor \neg A + \Box A \lor \diam \neg A$.
%The following holds:
%\begin{center}
%%$\logic \vd A$ if and only if $\WLL + 
%\end{center}
\end{theorem}
\begin{proof}
From left to right, it is easy to verify that all axioms of $\WLL$, as well as
$A \lor \neg A$ and $\Box A \lor \diam \neg A$, are derivable in $\logic$.
For the opposite direction, %the addition of $A \lor \neg A$ allows one to derive all theorem of $\CPL$,
adding $A \lor \neg A$ one derives as usual all theorems of $\CPL$,
while adding $\Box A \lor \diam \neg A$ one derives \axdual.
\end{proof}

Note that $\WLL + A \lor \neg A + \Box A \lor \diam \neg A$ is a proper extension of
$\WLL$, as it is stated by the following proposition.

%\begin{proposition}
%%For every W-logic $\WL$ the following hold:
%%\begin{itemize}
%%\item[$(i)$] $\WL \not\vd p \lor \neg p$.
%%\item[$(ii)$] $\WL \not\vd \Box p \lor \diam \neg p$.
%%\end{itemize}
%For every W-logic $\WL$, $\WL \not\vd p \lor \neg p$ and $\WL \not\vd \Box p \lor \diam \neg p$.
%\end{proposition}
%\begin{proof}
%%(i) Suppose that $\WL\vd A\lor B$. Then
%(i) If $\WL\vd A\lor B$, then
%by Theorem~\ref{th:eq seq ax}, $\seqWstar\vd \ \seq A \lor B$.
%%By inspecting the rules of $\seqWstar$
%The only rule of $\seqWstar$ with a consequence of the form $\seq A\lor B$ is \irlor,
%thus \irlor\ must be the last rule applied in the derivation, with premiss either $\seq A$ or $\seq B$.
%Then $\seqWstar\vd \ \seq A$ or $\seqWstar\vd \ \seq B$,
%therefore $\WL\vd A$ or $\WL\vd B$.
%
%(ii) If $\WL\vd \Box p \lor \diam \neg p$, then
%$\seqWstar\vd \ \seq \Box p \lor \diam \neg p$,
%thus $\seqWstar\vd \ \seq \Box p$ or $\seqWstar\vd \ \seq \diam\neg p$.
%However, by inspecting the rules of $\seqWstar$ it is easy to verify that
%$\seqWstar\not\vd \ \seq \Box p$ and $\seqWstar\not\vd \ \seq \diam\neg p$,
%therefore $\WL\not\vd \Box p \lor \diam \neg p$.
%\end{proof}

\begin{proposition}
For every W-logic $\WL$, $\WL \not\vd p \lor \neg p$ and $\WL \not\vd \Box p \lor \diam \neg p$.
\end{proposition}
\begin{proof}
If $\WL\vd p \lor \neg p$, then
by Theorem~\ref{th:eq seq ax}, $\seqWstar\vd \ \seq p \lor \neg p$.
%By inspecting the rules of $\seqWstar$
The only rule of $\seqWstar$ with a consequence of the form $\seq p \lor \neg p$ is \irlor,
thus \irlor\ must be the last rule applied in the derivation, with premiss either $\seq p$ or $\seq \neg p$.
However, by inspecting the rules of $\seqWstar$ it is easy to verify that
$\seqWstar\not\vd \ \seq p$ and $\seqWstar\not\vd \ \seq \neg p$,
therefore $\WL\not\vd  p \lor \neg p$.
$\WL \not\vd \Box p \lor \diam \neg p$
is proved similarly, observing that
$\seqWstar\not\vd \ \seq \Box p$ and $\seqWstar\not\vd \ \seq \diam\neg p$.
\end{proof}

In a similar way we can %also
 prove that W-logics 
satisfy the disjunction property.

\begin{proposition}[Disjunction property]
For all W-logics $\WL$ and all formulas $A$, $B$ of $\lan$,
if $\WL\vd A\lor B$, then $\WL\vd  A$ or $\WL\vd  B$.
\end{proposition}
\begin{proof}
If $\WL\vd A\lor B$, then
by Theorem~\ref{th:eq seq ax}, $\seqWstar\vd \ \seq A \lor B$.
%By inspecting the rules of $\seqWstar$
The only rule of $\seqWstar$ with a consequence of the form $\seq A\lor B$ is \irlor,
thus \irlor\ must be the last rule applied in the derivation, with premiss either $\seq A$ or $\seq B$.
Then $\seqWstar\vd \ \seq A$ or $\seqWstar\vd \ \seq B$,
therefore $\WL\vd A$ or $\WL\vd B$.
\end{proof}

%Finally, we can show
Furthermore, we can prove that derivability in W-logics is decidable.
To this aim, observe that for every rule $R$ of $\seqWstar$,
the premisses of $R$ have a smaller complexity than its conclusion,
with the only exceptions of \limp\ and \ruleiTbox\
which copy the principal formula into one premiss.
It follows that 
bottom-up proof search in $\seqWstar$ is not strictly terminating, 
however,
similarly to \cite[Ch.~4]{Troelstra:2000},
%termination of bottom-up proof search can be retrieved
termination can be gained 
%by applying simple loop-checking in order
%to avoid redundant applications of \limp\ and \ruleiTbox,
%preserving at the same time the completeness of the calculus.
by controlling the applications of \limp\ and \ruleiTbox\ with a simple loop-checking in order
to avoid redundant  applications of these rules,
preserving at the same time the completeness of the calculi.
Adopting this restriction, 
it turns out that every %derivation of a
proof tree for a root sequent $\G \seq \D$ is finite,
moreover there are only finitely many distinct proof trees for it.
Then, given the equivalence between $\seqWstar$ and $\WL$, it follows that 
derivability of $A$ in $\WL$ is decidable for any $A$:
%W-logics are decidable:
the decision procedure  trivially consists in checking all possible derivations of $\seq A$ in $\seqWstar$.

\begin{theorem}[Decidability]
Given a W-logic $\WL$ and a formula $A$ of $\lan$,
it is decidable whether $A$ is derivable in $\WL$.
\end{theorem}

Finally, we can prove that all W-logics % have the Craig interpolation property
enjoy Craig interpolation.
For every formula $A$ of $\lan$ and every multiset $\G = B_1, ..., B_n$,
%we denote $\var(A)$ 
%the set of all propositional variables occurring in $A$ 
%plus $\bot$,
%and denote $\var(\G)$ the set $\var(B_1) \cup ... \cup \var(B_n)$.
we define $\var(A) = \{\bot\} \cup \{p \in \atm \mid p \textup{ occurs in } A\}$,
and $\var(\G) = \var(B_1) \cup ... \cup \var(B_n)$.
Then Craig interpolation
% is defined as follows.
amounts to the following property.

\begin{definition}
%For any formula $A$ of $\lan$, %we define $\var(A)$ as 
%we denote $\var(A)$ 
%the set of all propositional variables occurring in $A$ 
%plus $\bot$. Then a
A logic $\WL$ %has the \emph{Craig interpolation} property
enjoys  \emph{Craig interpolation}
if for all $A,B\in\lan$,
if $\WL\vd A \imp B$, then there is $C\in\lan$ such that 
$\WL\vd A \imp C$,
$\WL\vd C \imp B$, 
and $\var(C)\subseteq\var(A)\cap\var(B)$.
\end{definition}

%We now prove that all W-logics % have the Craig interpolation property
%enjoy Craig interpolation.
The proof of Craig interpolation is based on the following lemma.

%As before, the proof that W-logics enjoy Craig interpolation
%is based on the sequent calculi, and amount

\begin{lemma}\label{lemma:craig}
For every calculus $\seqWstar$, if $\seqWstar\vd \G_1,\G_2 \seq \D$, 
 then there is $C\in\lan$ such that 
$\seqWstar\vd \G_1 \seq C$,
$\seqWstar\vd C, \G_2 \seq \D$, 
and $\var(C)\subseteq\var(\G_1)\cap\var(\G_2, \D)$.
\end{lemma}
\begin{proof}
By induction on the height $h$ of the derivation of $\G_1,\G_2 \seq \D$,
taking into account the last rule applied in the derivation.
If $h = 0$ or the last rule applied is propositional we refer to \cite{Ono}. 
%Here we consider a few most significant cases involving the modal rules,
Here we consider just one  significant case involving a modal rule,
for the other rules the proof is analogous.

%$\G_1,\G_2 \seq \D$ has the form $\G_1', \Box\G, \diam A, \G_2'\seq \diam B$
%%and it is obtained from $\G, A \seq B$.
%%There are four possible partitions of $\G_1', \Box\G, \diam A, \G_2'$ into  $\G_1,\G_2$.
Let \ruleiCdiam\ be the last rule applied in the derivation.
Then $\G_1,\G_2 \seq \D$ has the form $\G_1', \Box\G, \diam A, \G_2'\seq \diam B$
and it is obtained from the premiss $\G, A \seq B$.
There are four possible partitions of $\G_1', \Box\G, \diam A, \G_2'$ into  $\G_1,\G_2$.

\begin{itemize}
\item[(i)] $\G_1 = \G_1'$ and $\G_2 = \Box\G, \diam A, \G_2'$.
Then $\top$ is an interpolant: $\G_1' \seq \top$ is derivable, and from $\G, A \seq B$, by \ruleiCdiam\ we obtain $\top, \Box\G, \diam A, \G_2' \seq \diam B$.

\item[(ii)] $\G_1 = \G_1',\Box\G, \diam A$ and $\G_2 = \G_2'$.
By \ih, there is $C$ such that $\G, A \seq C$ and $C \seq B$ are derivable, and
 $\var(C)\subseteq\var(\G, A)\cap\var(B)$.
 Then by \ruleiCdiam\ we obtain $\G_1',\Box\G, \diam A \seq \diam C$ and $\diam C, \G_2' \seq \diam B$.
 Since $\var(\diam C) = \var(C)$, $\diam C$ is an interpolant.
\end{itemize}

The following two partitions are possible if $\G = D_1, ..., D_n$ and $n\geq 2$.
\begin{itemize}
\item[(iii)] $\G_1 = \G_1', \Box D_1, ..., \Box D_k$ and $\G_2 = \Box D_{k+1}, ..., \Box D_n, \diam A, \G_2'$
(for $1\leq k<n$).
By \ih, there is $C$ such that $D_1, ..., D_k \seq C$ and $C, D_{k+1}, ..., D_n, A \seq B$ are derivable, and
 $\var(C)\subseteq\var(D_1, ..., D_k)\cap\var(D_{k+1}, ..., D_n, A, B)$.
 Then by \ruleiCbox, $\G_1', \Box D_1, ..., \Box D_k \seq \Box C$ is derivable,
 and by \ruleiCdiam, $\Box C, \Box D_{k+1}, ..., \Box D_n, \diam A, \G_2' \seq \diam B$ is derivable.
 Then $\Box C$ is an interpolant.
 
\item[(iv)] $\G_1 = \G_1', \Box D_1, ..., \Box D_k, \diam A$ and $\G_2 = \Box D_{k+1}, ..., \Box D_n, \G_2'$
(for $1\leq k<n$).
By \ih, there is $C$ such that $D_1, ..., D_k, A \seq C$ and $C, D_{k+1}, ..., D_n \seq B$ are derivable, and
 $\var(C)\subseteq\var(D_1, ..., D_k, A)\cap\var(D_{k+1}, ..., D_n, B)$.
 Then by \ruleiCdiam, $\G_1', \Box D_1, ..., \Box D_k, \diam A \seq \diam C$ and
 $\diam C, \Box D_{k+1}, ..., \Box D_n, \G_2' \seq \diam B$ are derivable.
 Then $\diam C$ is an interpolant.
\end{itemize}
\end{proof}
%
%
%\begin{theorem}
%Let $\logic$ be any classical modal logic from Sec.~\ref{sec:prel},
%and $\WLL$ be the corresponding W-logic.
%Then $\logic$ is equivalent to $\WLL + A \lor \neg A + \Box A \lor \diam \neg A$.
%%The following holds:
%%\begin{center}
%%%$\logic \vd A$ if and only if $\WLL + 
%%\end{center}
%\end{theorem}

%Then Craig interpolation for W-logics easily follows.

\begin{theorem} %[Craig interpolation]
Every W-logic $\WL$ enjoys Craig interpolation.
\end{theorem}
\begin{proof}
Suppose that $\WL\vd A \imp B$.
Then $\seqWstar \vd A \seq B$.
By Lemma~\ref{lemma:craig},
there is $C\in\lan$ such that 
%$\seqWstar\vd A \seq C$,
%$\seqWstar\vd C \seq B$, 
%and $\var(C)\subseteq\var(A)\cap\var(B)$,
%therefore 
%$\WL\vd A \imp C$ and
%$\WL\vd C \imp B$.
$\var(C)\subseteq\var(A)\cap\var(B)$,
$\seqWstar\vd A \seq C$, and
$\seqWstar\vd C \seq B$,
thus
$\WL\vd A \imp C$ and
$\WL\vd C \imp B$.
\end{proof}

%\section{Definition via relational semantics}
\section{Semantics}\label{sec:semantics}
%\section{Semantics, and equivalence of the definitions}\label{sec:semantics}
We now define constructive \neigh\ models (\CNM s) that characterise the constructive modal logics defined in Sec.~\ref{sec:calculi}.
%Constructive \neigh\ models 
\CNM s
are defined analogously to \wij's %\const\ 
relational models \cite{wij}:
%VERSIONE VECCHIA
%\intu\ Kripke models are %endowed 
%enriched with a \neigh\ function
%(rather than a binary relation as in \wij's models),
%moreover the classical satisfaction clauses for modal formulas are generalised to %the
%all $\less$-successors,
%so that hereditariness is built into the clauses:
%in order that $w$ %forces
%satisfies $\Box A$, 
%it is required that 
we enrich \intu\ Kripke models with a \neigh\ function
(rather than a binary relation as in \wij's models),
moreover we generalise the classical satisfaction clauses for modal formulas to %the
all $\less$-successors,
so that hereditariness is built into the clauses:
in order that $w$ %forces
satisfies $\Box A$, 
we require that 
for all successors of $w$ there is a neighbourhood $\alpha$ such that $\alpha\ufor A$,
and similarly for $\diam A$.
We show that, for every classical logic %$\logic$ 
characterised by \neigh\ models 
satisfying some conditions from Def.~\ref{def:neigh model},
the corresponding W-logic %$\WLL$
 is characterised by the \CNM s
satisfying exactly the same conditions.
\CNM s are %formally 
defined as follows.

%%In the style of \wij's \const\ relational models, we enrich 
%%%classical \neigh\ models 
%%\neigh\ modes for classical logics
%%with a preorder $\less$ for the evaluation of intuitionistic implication.
%%%Moreover, 
%%%we generalise the forcing conditions for modal formulas 
%%%to all $\less$-successors:
%%At the same time, we build in hereditariness into the forcing conditions of modal formulas
%%by generalising them to all $\less$-successors:
%%in order that $w$ forces $\Box A$, the existence of a neighbourhood $\alpha$ such that $\alpha\ufor A$
%%must be true for all successors of $w$
%
%with the only difference that the forcing conditions for modal formulas (as well as for $\imp$)
%are now generalised %parametrised
%to all successors:
%in order that %$w\Vd\Box A$, 
%$w$ forces $\Box A$, the existence of a neighbourhood $\alpha$ such that $\alpha\ufor A$
%must be true for all successors of $w$.
%
%
%In the previous section we have defined some intuitionistic and minimal modal logics
%via natural restrictions on sequent calculi for classical modal logics. 
%
%%While the two definitions are independent one of the the other, 
%%they return exactly the same logics.
%We now show that, despite being independent one of the other, 
%the two definitions 
%(i.e.~based on the sequent calculi and based on the semantics) return exactly the same logics.

\begin{definition}\label{def:semantics}
A \emph{\const\ \neigh\ model} (\CNM)  is a tuple $\M = \langle \W, \less, \N, \V\rangle$, 
where 
%$\W$ is a non-empty set of items, called worlds, 
$\W$ is a non-empty set of worlds, 
%$\less$ is a reflexive and transitive relation on $\W$, 
$\less$ is a preorder on $\W$,  
$\N: \W\longrightarrow \pow(\pow(\W))$ is a neighbourhood function, and
$\V : \atm \longrightarrow \pow(\W)$ is a hereditary valuation function
(\ie, if $w\in\V(p)$ and $w\less v$, then $v\in\V(p)$).
The forcing relation $\M, w\Vd A$ is defined as in Def.~\ref{def:neigh model} for 
%$A = p, \bot, B \land C, B \lor C, B \imp C$, while for modal formulas it is as follows:
$A = p, \bot, B \land C, B \lor C$, otherwise it is as follows:
\begin{center}
\begin{tabular}{lllll}
$\M, w \Vd B \imp C$ & iff & for all $v\more w$, $\M, v \Vd B$ implies $\M, v \Vd C$. \\
$\M, w\Vd\Box B$ & iff & for all $v\more w$, there is $\alpha\in\N(v)$ such that $\alpha\ufor B$. \\
$\M, w\Vd\diam B$ & iff & for all $v\more w$, for all $\alpha\in\N(v)$, $\alpha\efor B$. \\ 
\end{tabular}
\end{center}
We consider the following properties on \CNM s:
\begin{center}
\begin{tabular}{llll}
(\cC) \ If $\alpha,\beta\in\N(w)$, then $\alpha\cap\beta\in\N(w)$. &&
(\cN) \ $\N(w)\not=\emptyset$. \\
(\cD) \  If $\alpha,\beta\in\N(w)$, then $\alpha\cap\beta\not=\emptyset$. &&
(\cP) \  $\emptyset\notin\N(w)$. \\
(\cT) \ If $\alpha\in\N(w)$, then $w\in\alpha$. \\
\end{tabular}
\end{center}
We say that $\M$ is a model for a logic $\Wlogic$, or is a $\Wlogic$-model,
%if it satisfies the condition (\cX) for every modal axiom \varaxbox\ of $\Wlogic$
%(among \axCbox, \axNbox, \axTbox, \axD, \axPbox).
%if it satisfies the condition (\cX) for every modal axiom \varax\ of $\Wlogic$
%(among \axCbox, \axNbox, \axTbox, \axD, \axPdiam).
if for every modal axiom \varax\ of $\Wlogic$
(among \axCbox, \axNbox, \axTbox, \axD, \axPdiam),
$\M$ satisfies the corresponding condition (\cX).
%We say that $\M$ is a model for a logic $\Wlogic$,
%or is a $\Wlogic$-model, if it satisfies the semantics conditions among (\cC), (\cN), (\cT), (\cD), (\cP) 
%from Def.~\ref{def:neigh model} of 
%the \neigh\ models for the corresponding classical logic $\logic$. % (Def.~\ref{def:neigh model}).
\end{definition}

\CNM s %are the simplest possible combination of
represent the simplest way of combining \intu\ Kripke models and \neigh\ models.
From the definition of $\V$ and of the satisfaction clauses, it immediately follows that 
\CNM s %satisfy
enjoy the hereditary property. 

\begin{proposition}[Hereditary property]\label{prop:her}
For all $A\in\lan$ and all \CNM s $\M$, if $\M,w\Vd A$ and $w\less v$, then $\M,v\Vd A$.
\end{proposition}
\begin{proof}
By induction on the construction of $A$.
%Assume $w \less v$.
%%
%($A = p$) If $w \Vd p$, then $w\in\V(p)$, then  by Def.~\ref{def:semantics}, $v\in\V(p)$, thus $v\Vd p$.
%%
%($A = \bot$) $w\not\Vd \bot$.
%%
%($A = B \land C$) If $w\Vd B \land C$, then $w\Vd B$ and $w\Vd C$, then by \ih, $v\Vd B$ and $v\Vd C$,
%thus $v\Vd B\land C$.
%%
%($A = B \lor C$)
%If $w\Vd B \lor C$, then $w\Vd B$ or $w\Vd C$, then by \ih, $v\Vd B$ or $v\Vd C$, thus $v\Vd B\lor C$.
%%Analogous to $A = B \land C$.
%%
%($A = B \imp C$) If $w\Vd B\lor C$, then for all $u\more w$, $u\Vd B$ implies $u\Vd C$,
%then for all $u\more v$, $u\Vd B$ implies $u\Vd C$, thus $v\Vd B\imp C$.
%%
We only consider the inductive cases $A= \Box B, \diam B$ as the other cases are standard.
($A = \Box B$) If $w\Vd \Box B$, then for all $u\more w$, there is $\alpha\in\N(u)$ such that $\alpha\ufor B$,
then for all $u\more v$, there is $\alpha\in\N(u)$ such that $\alpha\ufor B$, thus $v\Vd \Box B$.
($A = \diam B$) If $w\Vd \diam B$, then for all $u\more w$, for all $\alpha\in\N(u)$, $\alpha\efor B$,
then for all $u\more v$,  for all $\alpha\in\N(u)$, $\alpha\efor B$, thus $v\Vd \diam B$.
\end{proof}
%\noindent
%%\textbf{Proof sketch.}
%\textbf{Sketch of Proof.}
%By induction on the construction of $A$ (see Appendix B).
%\qed\medskip

We show that, for any classical modal logic characterised by \neigh\ models
satisfying some conditions among (\cC), (\cN), (\cT), (\cD), (\cP),
the corresponding W-logic is characterised by the \CNM s satisfying the same conditions.
%\nb{detto praticamente uguale nell'intro della Sez.}
We first show that W-logics are sound with respect to their classes of \CNM s,
then prove their completeness by a canonical model construction.

%From the definition of $\V$ and of the satisfaction clauses, it immediately follows that 
%\CNM s %satisfy
%enjoy the hereditary property. 
%
%\begin{proposition}[Hereditary property]\label{prop:her}
%For all $A\in\lan$ and \CNM\ $\M$, if $\M,w\Vd A$ and $w\less v$, then $\M,v\Vd A$.
%\end{proposition}

\begin{theorem}[Soundness]\label{th:soundness}
For every W-logic $\Wlogic$, 
if $\Wlogic\vd A$, then $\M\models A$ 
%for all \const\ neighbourhood models $\M$ for $\Wlogic$. 
for all $\Wlogic$-models $\M$.
\end{theorem}
\begin{proof}
%\noindent
%\textbf{Sketch of Proof.}
As usual, we need to show that %all 
the axioms of $\Wlogic$ are valid in all $\Wlogic$-models,
and that the rules of $\Wlogic$ preserve the validity in $\Wlogic$-models.
%(Proof in % the 
%Appendix B).
%\qed

\begin{enumerate}[leftmargin=*, align=left]
\item[(\ruleMbox)] %Suppose 
Assume that $\M\models A \imp B$ and $w\Vd \Box A$.
Then for all $v\more w$, there is $\alpha\in\N(v)$ such that $\alpha\ufor A$.
If follows that $\alpha\ufor B$. Therefore $w\Vd \Box B$.

\item[(\ruleMdiam)] %Suppose 
Assume that  $\M\models A \imp B$ and $w\Vd \diam A$.
Then for all $v\more w$, for all $\alpha\in\N(v)$, $\alpha\efor A$.
If follows that $\alpha\efor B$. Therefore $w\Vd \diam B$.

\item[(\axdualand)] %Suppose 
Assume that  $w\Vd\Box A \land \diam\neg A$.
Then for all $v\more w$, there is $\alpha\in\N(v)$ such that $\alpha\ufor A$, and for all $\beta\in\N(v)$, $\beta\efor \neg A$.
Thus there is $\gamma\in\N(w)$ such that $\gamma\efor A\land \neg A$, which is impossible.
Therefore $\M\models\neg(\Box A \land \diam\neg A)$.

\item[(\axCbox)] %Suppose 
Assume that  $\M$ satisfies condition (\cC) and $w\Vd\Box A \land \Box B$.
Then for all $v\more w$, there is $\alpha\in\N(v)$ such that $\alpha\ufor A$,
and there is $\beta\in\N(v)$ such that $\beta\ufor B$.
By (\cC), $\alpha\cap\beta\in\N(v)$, moreover $\alpha\cap\beta\ufor A\land B$.
Thus $w\Vd\Box(A\land B)$.

\item[(\axKdiam)] %Suppose 
Assume by contradiction that $\M$ satisfies %condition
 (\cC), $w\Vd\Box (A \imp B)$, $w \Vd \diam A$
and  $w\not\Vd\diam B$.
By $w\not\Vd\diam B$, there are $v\more w$ and $\alpha\in\N(v)$ such that $\alpha\not\efor B$.
Then by $w\Vd\Box (A \imp B)$, there is $\beta\in\N(v)$ such that $\beta\ufor A \imp B$.
Thus by (\cC), $\alpha\cap\beta\in\N(v)$.
It follows $\alpha\cap\beta\not\efor B$ and $\alpha\cap\beta\ufor A\imp B$.
However by $w \Vd \diam A$,  $\alpha\cap\beta\efor A$,
thus $\alpha\cap\beta\efor B$, which gives a contradiction.

\item[(\axNbox)]
If $\M$ satisfies the condition (\cN), then for all $w$ there is $\alpha\in\N(w)$. Moreover, $\alpha\ufor\top$,
therefore $\M\models\Box\top$.

%(\axNbox, \axNdiam) If $\M$ satisfies condition (\cN), then for all $w$ there is $\alpha\in\N(w)$. Moreover, $\alpha\ufor\top$,
%therefore $\M\models\Box\top$. In addition, it is never the case that for all $\alpha\in\N(w)$, $\alpha\efor\bot$.
%Therefore $\M\models\neg\diam\bot$.

\item[(\axPdiam)]
If $\M$ satisfies the condition (\cP), then for all $w$ and
all $\alpha\in\N(v)$, $\alpha\not=\emptyset$, thus $\alpha\efor \top$. Therefore $\M\models\diam \top$.

%(\axPbox, \axPdiam) If $\M$ satisfies the condition (\cP), then for all $w$ and
%all $\alpha\in\N(v)$, $\alpha\not=\emptyset$, thus $\alpha\efor \top$. Therefore $\M\models\diam \top$.
%Moreover, there is no $\alpha\in\N(v)$ such that $\alpha\ufor\bot$. Therefore $\M\models\neg\Box\bot$.

\item[(\axD)] %Suppose 
Assume by contradiction that $\M$ satisfies the condition
 (\cD), $w\Vd\Box A$, and  $w\not\Vd\diam A$.
Then there are $v\more w$ and $\alpha\in\N(v)$ such that $\alpha\not\efor A$.
Moreover, there is $\beta\in\N(v)$ such that $\beta\ufor A$.
By (\cD), there is $u\in\alpha\cap\beta$. Thus $u\not\Vd A$ and $u\Vd A$.
Therefore $w\Vd\diam A$.

\item[(\axTbox, \axTdiam)]
Suppose that $\M$ satisfies the condition (\cT). Then if $w\Vd\Box A$, 
 there is $\alpha\in\N(w)$ such that $\alpha\ufor A$. By (\cT), $w\in\alpha$, thus $w\Vd A$.
Moreover, if $w\Vd A$, then by Prop.~\ref{prop:her}, $v\Vd A$ for all $v\more w$.
By (\cT), $v\in\alpha$ for all $\alpha\in\N(v)$, thus $\alpha\efor A$.
Therefore $w\Vd\diam A$.
\end{enumerate}
\end{proof}

\subsection{Completeness}

%\medskip
%\noindent
%\textbf{Completeness.}

We now prove that W-logics are complete with respect to the corresponding classes of \CNM s.
As usual, for every logic $\Wlogic$, we call $\WL$-\emph{prime} any set $\Sigma$ of formulas of $\lan$ 
such that 
%(i) 
$\Sigma\not\vd_{\WL}\bot$ (consistency), %($\WL$-consistency),
%(ii) 
if $\Sigma\vd_{\WL} A$, then $A\in\Sigma$ (closure under derivation), and 
%(iii) 
if $A\lor B\in\Sigma$, then $A\in\Sigma$ or $B\in\Sigma$ (disjunction property).
Moreover, for every set of formulas $\Sigma$, 
we denote $\boxm\Sigma$ the set $\{A \mid \Box A\in\Sigma\}$.
One can prove in a standard way the following lemma.

\begin{lemma}[Lindenbaum]\label{lemma:lind}
If $\Sigma\not\vd_{\WL} A$, then there is a $\WL$-prime set $\Pi$ such that $\Sigma\subseteq\Pi$ and $A\notin\Pi$.
\end{lemma}

We also consider the following notion of segment 
(we adopt the terminology of \cite{wij}),
%(the terminology comes from \cite{wij}).
%\footnote{We adopt the terminology of \cite{wij}.}
and prove the subsequent lemma that will be needed in the following.

\begin{definition}\label{def:segm}
For every logic $\Wlogic$, 
a $\WL$-\emph{\seg} %\nb{definire WL-segment: occhio alle propr che servono es per T ecc}
is a pair $(\Sigma, \CC)$,
where $\Sigma$ is a $\WL$-prime set, and $\CC$ is a class of sets of $\WL$-prime sets such that:
\begin{itemize}
\item if $\Box A\in\Sigma$, then there is $\U\in\CC$ such that for all $\Pi\in\U$, $A\in\Pi$; and
\item if $\diam A\in\Sigma$, then for all $\U\in\CC$, there is $\Pi\in\U$ such that $A\in\Pi$.
\end{itemize}
$\WCstar$-, $\WDstar$- and $\WTstar$-segments must satisfy also the following %additional 
conditions:
($\WCstar$) If $\U,\U'\in\CC$, then $\U\cap\U'\in\CC$.
($\WDstar$) If $\U,\U'\in\CC$, then $\U\cap\U'\not=\emptyset$.
($\WTstar$) For all $\U\in\CC$, $\Sigma\in\U$.
\end{definition}

\begin{lemma}\label{lemma:segm}
For every $\WL$-prime set $\Sigma$, there exists a $\WL$-segment $(\Sigma, \CC)$.
\end{lemma}
\begin{proof}
Given a $\WL$-prime set $\Sigma$, we construct a $\WL$-segment $(\Sigma, \CC)$ as follows.
If there is no $\Box A\in\Sigma$, we put $\CC=\emptyset$.
If there is no $\diam A\in\Sigma$, we put $\CC = \{\emptyset\}$.
Otherwise we distinguish two cases.

\begin{itemize}
\item[(i)]
%(i) 
$\WL$ does not contain \axCbox, \axKdiam.
Let $\Box A, \diam B\in\Sigma$. Then $A,B\not\vd_{\WL}\bot$
(otherwise by \Rdualand, $\Box A, \diam B\vd_{\WL}\bot$, against the consistency of $\Sigma$).
Then by Lemma~\ref{lemma:lind}, there is $\Sab$ $\WL$-prime such that $A,B\in\Sab$.
%For all $\Box A\in\Sigma$ we define $\Ua = \{\Sab \mid \diam B\in\Sigma\}$,
%moreover $\CC = \{\Ua \mid \Box A \in\Sigma\}$.
For all $\Box A\in\Sigma$, 
we define $\Ua = \{\Sab \mid \diam B\in\Sigma\}$
if $\WL$ does not contain \axTbox, and 
$\Ua = \{\Sab \mid \diam B\in\Sigma\} \cup\{\Sigma\}$  if it contains \axTbox. %otherwise.
Moreover we define $\CC = \{\Ua \mid \Box A \in\Sigma\}$.
Then $(\Sigma, \CC)$ is a $\WL$-segment:
if $\Box A\in\Sigma$, then $\Ua\in\CC$ and for all %$\Sab\in\Ua$, $A\in\Sab$.
$\Sigma'\in\Ua$, $A\in\Sigma'$.
If $\diam B\in\Sigma$, then for all $\Ua\in\CC$, there is $\Sab\in\Ua$ such that $B\in\Sab$.
Moreover for $\WDstar$, if $\Ua,\Ub\in\CC$, $\Ua\not=\Ub$, 
then $\Box A,\Box B\in\Sigma$,
then by axiom \axD, $\diam A,\diam B\in\Sigma$,
thus $\Sab\in\Ua\cap\Ub$.

\item[(ii)]
%(ii) 
$\WL$ contains \axCbox, \axKdiam.
Let $\diam B\in\Sigma$. Then $\boxm\Sigma\cup\{B\}\not\vd_{\WL}\bot$
(otherwise by \Rdualand\ and \axCbox, %$\Sigma\cup\{\diam B\}\vd_{\WL}\bot$, i.e.~$\Sigma\vd_{\WL}\bot$).
%$\Sigma\cup\{\diam B\}\vd_{\WL}\bot$).
$\Sigma\vd_{\WL}\bot$).
Then by Lemma~\ref{lemma:lind}, there is $\Sb$ $\WL$-prime such that $\boxm\Sigma\subseteq\Sb$ and $B\in\Sb$.
We define $\U = \{\Sb \mid \diam B\in\Sigma\}$
if $\WL$ does not contain \axTbox, and 
$\U = \{\Sb \mid \diam B\in\Sigma\} \cup\{\Sigma\}$ if it contains \axTbox. % otherwise.
Moreover we define $\CC = \{\U\}$.
It is easy to verify that $(\Sigma, \CC)$ is a $\WCstar$-segment.
\end{itemize}
\end{proof}

We consider the following definition of canonical model.

\begin{definition}\label{def:can model}
For every logic $\Wlogic$, 
the \emph{canonical model} for $\WL$ is the tuple $\Mc = \langle \Wc, \lessc, \Nc, \Vc \rangle$,
where:
\begin{itemize}
\item $\Wc$ is the class of all $\WL$-segments;
%(for $\WCstar$, $\WDstar$, or $\WTstar$, it is the class of $\WCstar$-, resp.~$\WDstar$-, resp.~$\WTstar$-segments);
%non serve specificarlo, è già implicito nella def di segment
\item $(\Sigma, \CC) \lessc (\Sigma', \CC')$ %iff
if and only if $\Sigma\subseteq\Sigma'$;
\item for every set $\U$ of $\WL$-prime sets, $\aU = \{(\Sigma, \CC) \mid \Sigma\in\U\}$;
\item $\aU\in\Nc((\Sigma,\CC))$ %iff
if and only if $\U\in\CC$;
\item $(\Sigma,\CC)\in\Vc(p)$ if and only if $p\in\Sigma$.
\end{itemize}
\end{definition}

%\tiz{From the definition of segment it immediately follows that the canonical model for $\WL$ is a $\WL$-model.\nb{verificare}}

%We prove the following lemmas.
%We prove the following two lemmas, from which completeness of $\WL$-logics follows in the standard way.
We prove the following two lemmas which entail completeness of $\WL$-logics.

\begin{lemma}\label{lemma:model lemma}
The canonical model for $\WL$ is a \CNM\ for $\WL$.
\end{lemma}
\begin{proof}
We show that the canonical model for $\WL$ satisfies the conditions of \CNM s for $\WL$.
(\cC), (\cD), (\cT) and hereditariness of $\Vc$ are immediate by 
Defs.~\ref{def:segm} and \ref{def:can model}.
(\cN)
For all $\WNstar$-prime sets $\Sigma$, $\Box\top\in\Sigma$,
then for all $\WNstar$-segments $(\Sigma,\CC)$, $\CC\not=\emptyset$,
thus $\Nc((\Sigma,\CC))\not=\emptyset$.
(\cP)
For all $\WPstar$-prime sets $\Sigma$, $\diam\top\in\Sigma$,
then for all $\WPstar$-segments $(\Sigma,\CC)$, for all $\U\in\CC$, $\U\not=\emptyset$,
thus $\emptyset\notin\Nc((\Sigma,\CC))$.
%
%(\cC), (\cD) and (\cT) immediate by definition of segment.
%(\cC), (\cD), (\cT) and hereditariness of $\Vc$ are immediate by 
%%definition of segment and canonical model.
%Defs.~\ref{def:segm}, \ref{def:can model}.
\end{proof}

%We prove the following truth lemma.

\begin{lemma}\label{lemma:truth lemma}
Let $\WL$ be a W-logic, and $\Mc = \langle \Wc, \lessc, \Nc, \Vc \rangle$
be the canonical model for $\WL$. Then for
%For 
all $(\Sigma,\CC)\in\Wc$, 
$(\Sigma,\CC)\Vd A$ if and only  if $A\in\Sigma$.
\end{lemma}
\begin{proof}
By induction on the construction of $A$.
If $A = p$ or $A = \bot$, the proof is immediate by definition of $\Vc$ or by consistency of $\Sigma$,
moreover for $A = B\land C, B\lor C$ the proof is immediate by applying the inductive hypothesis.
We consider the 
%other
remaining 
cases.
%We consider the inductive cases.
%($A = B\land C$) $(\Sigma,\CC)\Vd B\land C$ iff $(\Sigma,\CC)\Vd B$ and $(\Sigma,\CC)\Vd C$
%iff (\ih)  $B\in\Sigma$ and $C\in\Sigma$ iff $B\land C\in\Sigma$.
%%
%($A = B\lor C$) $(\Sigma,\CC)\Vd B\lor C$ iff $(\Sigma,\CC)\Vd B$ or $(\Sigma,\CC)\Vd C$
%iff (\ih)  $B\in\Sigma$ or $C\in\Sigma$ iff $B\lor C\in\Sigma$.

%\begin{list}{$\bullet$}{\leftmargin=1em \itemindent=0em}
%\item 
\begin{list}{}{\leftmargin=0em \itemindent=0em}
\item
%($A = B \imp C$) 
$\bullet$ $A = B \imp C$: 
%$\boxed{A = B \imp C}$
%\begin{addmargin}[1em]{0em}% 1em left, 2em right
If $B\imp C\in\Sigma$, then suppose $(\Sigma, \CC)\lessc(\Sigma',\CC')$ and $(\Sigma',\CC')\Vd B$.
Then $\Sigma\subseteq\Sigma'$, thus $B\imp C\in\Sigma'$.
Moreover by \ih, $B\in\Sigma'$, then $C\in\Sigma'$, thus by \ih, $(\Sigma',\CC')\Vd C$.
Therefore $(\Sigma, \CC)\Vd B \imp C$.
If instead $B\imp C\notin\Sigma$, then $\Sigma\not\vd B\imp C$, thus $\Sigma\cup\{B\}\not\vd C$.
By Lemma~\ref{lemma:lind}, there is $\Sigma'$ 
$\WL$-prime such that $\Sigma\cup\{B\}\subseteq\Sigma'$ and $C\notin\Sigma'$.
Then by Lemma~\ref{lemma:segm} and Def.~\ref{def:can model}, there is a $\WL$-segment $(\Sigma',\CC')\in\Wc$.
By definition, $(\Sigma, \CC)\lessc(\Sigma', \CC')$, and by \ih, $(\Sigma', \CC')\Vd B$ and $(\Sigma', \CC')\not\Vd C$.
Therefore $(\Sigma, \CC)\not\Vd B\imp C$.
%\end{enumerate}
%\end{list}

%\begin{enumerate}[leftmargin=*, align=left, leftmargin=1em, itemindent=1em]
%	\item[$A = B \imp C$]  blablabla
%	blablabla blablablablablablav  blablablablablabla blablabla v blablabla blablablablablablablablabla blablabla blablablablablabla blablablablablabla blablabla vblablabla
%\end{enumerate}

%($A = \Box B$) If $\Box B\in\Sigma$, then for all $(\Sigma',\CC')\morec(\Sigma,\CC)$, $\Box B\in\Sigma'$.
%By definition of segment, there is $\U'\in\CC'$ such that for all $\Sigma''\in\U'$, $B\in\Sigma''$.
%Then by \ih, $(\Sigma'',\CC'')\Vd B$ for all $(\Sigma'',\CC'')\in\Wc$, thus $\aUp\in\N((\Sigma',\CC'))$ and $\aUp\ufor B$.
%It follows $(\Sigma, \CC)\Vd \Box B$. 
%%
%If instead $\Box B\notin\Sigma$, 
%we distinguish two cases:

%\smallskip

\item
$\bullet$ $A = \Box B$:
%($A = \Box B$) 
If $\Box B\in\Sigma$, then for all $(\Sigma',\CC')\morec(\Sigma,\CC)$, $\Box B\in\Sigma'$.
By definition of segment,
 there is $\U'\in\CC'$ such that for all $\Sigma''\in\U'$, $B\in\Sigma''$.
Then $\aUp\in\N((\Sigma',\CC'))$, moreover by \ih, $(\Sigma'',\CC'')\Vd B$ for all $(\Sigma'',\CC'')\in\aUp$.
Therefore $(\Sigma, \CC)\Vd \Box B$. 
Now suppose that $\Box B\notin\Sigma$.
If there is no $\Box C\in\Sigma$, then $(\Sigma,\emptyset)$ is a $\WL$-segment, 
moreover $(\Sigma,\CC)\lessc(\Sigma,\emptyset)$ and $\Nc((\Sigma,\emptyset)) = \emptyset$,
thus $(\Sigma,\CC)\not\Vd\Box B$.
If instead there is $\Box C\in\Sigma$,
we distinguish two cases:
\end{list}

%\begin{list}{(i)}{\leftmargin=1em \itemindent=0em}
%(i) 
\begin{itemize}
\item[(i)] $\WL$ does not contain \axCbox, \axKdiam.
Then for all $\Box D\in\Sigma$, $D\not\vd B$
(otherwise by \monbox, $\Box D \vd \Box B$, whence $\Box B\in\Sigma$).
Then there is $\Sd$ $\WL$-prime such that $D\in\Sd$ and $B\notin\Sd$.
Moreover, for all $\diam C\in\Sigma$, $C,D\not\vd\bot$
(otherwise by \Rdualand, $\diam C, \Box D\vd \bot$, whence $\bot\in\Sigma$).
Then there is $\Scd$ $\WL$-prime such that $C,D\in\Scd$.
For all $\Box D\in\Sigma$, we define 
$\Ud = \{\Sd\} \cup \{\Scd \mid \diam C\in\Sigma\}$
if $\WL$ does not contain \axTbox,
and $\Ud = \{\Sd\} \cup \{\Scd \mid \diam C\in\Sigma\} \cup \{\Sigma\}$  if it contains \axTbox. %otherwise.
Moreover, we define
$\CC' = \{\Ud \mid \Box D\in\Sigma\}$.
It is easy to verify that $(\Sigma, \CC')$ is a $\WL$-segment.
Moreover, for all $\Ud\in\CC'$, $\Sd\in\Ud$ and $B\notin\Sd$,
thus by \ih, $(\Sd, \CC'')\not\Vd B$ for any $(\Sd, \CC'')\in\Wc$.
It follows that for all $\aU\in\Nc((\Sigma, \CC'))$, $\aU\not\ufor B$.
Thus $(\Sigma, \CC')\not\Vd\Box B$, and since $(\Sigma, \CC)\lessc (\Sigma, \CC')$,
$(\Sigma, \CC)\not\Vd\Box B$.

\item[(ii)] 
$\WL$ contains \axCbox, \axKdiam.
Then $\boxm\Sigma\not\vd B$
(otherwise by %\ruleiCbox\ or \ruleiKbox, 
\monbox\ and \axCbox,
$\Sigma\vd\Box B$),
then there is $\Sigma'$ $\WCstar$-prime such that $\boxm\Sigma\subseteq\Sigma'$ and $B\notin\Sigma'$.
Moreover, for all $\diam C\in\Sigma$, $\boxm\Sigma\cup\{C\}\not\vd\bot$, 
%(otherwise by \Rdualand, $\diam C, \Box D\vd \bot$, whence $\bot\in\Sigma$).
then there is $\Sc$ $\WCstar$-prime such that $\boxm\Sigma\subseteq\Sc$ and $C\in\Sc$.
We define 
$\U' = \{\Sigma'\} \cup \{\Sc \mid \diam C\in\Sigma\}$
if $\WCstar$ does not contain \axTbox,
and $\U' = \{\Sigma'\} \cup \{\Sc \mid \diam C\in\Sigma\} \cup \{\Sigma\}$  if it contains \axTbox. % otherwise.
Moreover, we define
$\CC' = \{\U'\}$.
It is easy to verify that $(\Sigma, \CC')$ is a $\WCstar$-segment.
Moreover, since $B\notin\Sigma'$,
by \ih, $(\Sigma', \CC'')\not\Vd B$ for any $(\Sigma', \CC'')\in\Wc$,
it follows that for all $\aU\in\Nc((\Sigma, \CC'))$, $\aU\not\ufor B$.
Thus $(\Sigma, \CC')\not\Vd\Box B$, and since $(\Sigma, \CC)\lessc (\Sigma, \CC')$,
$(\Sigma, \CC)\not\Vd\Box B$.
\end{itemize}

\begin{list}{}{\leftmargin=0em \itemindent=0em}
\item
$\bullet$ $A = \diam B$:
%($A = \diam B$) 
If $\diam B\in\Sigma$, then for all $(\Sigma',\CC')\morec(\Sigma,\CC)$, $\diam B\in\Sigma'$.
%By definition of segment, for all $\U'\in\CC'$, there is $\Sigma''\in\U'$ such that $B\in\Sigma''$,
%thus by \ih, $(\Sigma'',\CC'')\Vd B$ for all $(\Sigma'',\CC'')\in\Wc$. 
%Then for all $\aUp\in\Nc((\Sigma',\CC'))$, there is $(\Sigma'',\CC'')\in\aUp$ such that $(\Sigma'',\CC'')\Vd B$.
%It follows that $(\Sigma,\CC)\Vd\diam B$.
By definition of segment, for all $\U'\in\CC'$, there is $\Sigma''\in\U'$ such that $B\in\Sigma''$.
Then for all $\aUp\in\Nc((\Sigma',\CC'))$, there is $(\Sigma'',\CC'')\in\aUp$ such that $B\in\Sigma''$,
thus by \ih, $(\Sigma'',\CC'')\Vd B$.
It follows that $(\Sigma,\CC)\Vd\diam B$.
Now suppose that $\diam B\notin\Sigma$.
If there is no $\diam C\in\Sigma$, then $(\Sigma,\{\emptyset\})$ is a $\WL$-segment, 
moreover $(\Sigma,\CC)\lessc(\Sigma,\{\emptyset\})$ and $\Nc((\Sigma,\{\emptyset\})) = \{\emptyset\}$,
thus $(\Sigma,\CC)\not\Vd\diam B$.
If instead there is $\diam C\in\Sigma$,
we distinguish two cases:
\end{list}

\begin{itemize}
\item[(i)] 
$\WL$ does not contain \axCbox, \axKdiam.
Then for all $\diam C\in\Sigma$, $C\not\vd B$ 
(otherwise by \mondiam, $\diam C \vd \diam B$, whence $\diam B\in\Sigma$).
Then there is $\Sc$ $\WL$-prime such that $C\in\Sc$ and $B\notin\Sc$.
Moreover, for all $\Box D\in\Sigma$, $C,D\not\vd\bot$
(otherwise by \Rdualand, $\diam C, \Box D\vd \bot$, whence $\bot\in\Sigma$).
Then there is $\Scd$ $\WL$-prime such that $C,D\in\Scd$.
If in addition $\WL$ does not contain \axTbox, we define 
$\U' = \{\Sc \mid \diam C\in\Sigma\}$, and
for all $\Box D\in\Sigma$, $\Ud = \{\Scd \mid \diam C\in\Sigma\}$;
otherwise we define 
$\U' = \{\Sc \mid \diam C\in\Sigma\} \cup\{\Sigma\}$, and
for all $\Box D\in\Sigma$, $\Ud = \{\Scd \mid \diam C\in\Sigma\} \cup\{\Sigma\}$.
Moreover, we define
$\CC' = \{\U'\} \cup \{\Ud \mid \Box D\in\Sigma\}$.
It is easy to verify that $(\Sigma, \CC')$ is a $\WL$-segment:
for instance for $\WDstar$, 
if $\Ud,\Ue\in\CC'$, then $\Ud\cap\Ue\not=\emptyset$
(cf.~proof of Lemma~\ref{lemma:segm}), moreover 
%since by \axPdiam, $\diam\top\in\Sigma$,
by \axPdiam, $\diam\top\in\Sigma$, thus
for every $\Box D\in\Sigma$, $\Stopd\in\Ud\cap\U'$, then
$\Ud\cap\U'\not=\emptyset$.
In addition, by definition we have $\aUp\in\Nc((\Sigma,\CC'))$, and for all $\Sigma'\in\U'$, $B\notin\Sigma'$
(in particular, by \axTdiam, $B\notin\Sigma$).
Thus by \ih, for all $\Sigma'\in\U'$ and all $(\Sigma',\CC'')\in\Wc$, $(\Sigma',\CC'')\not\Vd B$,
then $\aUp\not\efor B$. Therefore $(\Sigma, \CC')\not\Vd\diam B$, and since $(\Sigma, \CC)\lessc (\Sigma, \CC')$,
$(\Sigma, \CC)\not\Vd\diam B$.

\item[(ii)] 
$\WL$ contains \axCbox, \axKdiam.
Then for all $\diam C\in\Sigma$, $\boxm\Sigma\cup\{C\}\not\vd B$ 
(otherwise by %\ruleiCdiam\ or \ruleiKdiam, 
\monbox, \axCbox\ and \axKdiam,
$\Sigma\vd\diam B$).
Then there is $\Sc$ $\WCstar$-prime such that $\boxm\Sigma\cup\{C\}\subseteq\Sc$ and $B\notin\Sc$.
We define 
$\U' = \{\Sc \mid \diam C\in\Sigma\}$ if $\WCstar$ does not contain \axTbox, and
$\U' = \{\Sc \mid \diam C\in\Sigma\} \cup\{\Sigma\}$  if it contains \axTbox. %otherwise.
Moreover we define $\CC' = \{\U'\}$.
It is easy to verify that $(\Sigma, \CC')$ is a $\WCstar$-segment.
%Moreover, for all $\Sc\in\U'$ and all $(\Sc,\CC'')\in\Wc$, 
Moreover, for all $\Sigma'\in\U'$, $B\notin\Sigma'$, then by \ih, for all $(\Sigma',\CC'')\in\Wc$, 
$(\Sigma',\CC'')\not\Vd B$,
thus $\aUp\not\efor B$. It follows that $(\Sigma, \CC')\not\Vd\diam B$, and since $(\Sigma, \CC)\lessc (\Sigma, \CC')$,
$(\Sigma, \CC)\not\Vd\diam B$.
\end{itemize}
\end{proof}

\begin{theorem}[Completeness]
For every W-logic $\Wlogic$, 
if $\M\models A$ for all $\Wlogic$-models $\M$,
then $\Wlogic\vd A$.
\end{theorem}
\begin{proof}
Suppose that
%Assume 
$\Wlogic\not\vd A$.
Then by Lemma~\ref{lemma:lind}, there is a $\WL$-prime set $\Sigma$ such that $A\notin\Sigma$,
thus by Lemma~\ref{lemma:segm}, 
%there exists a $\WL$-segment $(\Sigma,\CC)$.
%By Def.~\ref{def:can model}, $(\Sigma,\CC)$ belongs to the canonical model $\Mc$ for $\WL$, 
%and by Lemma~\ref{lemma:truth lemma}, $(\Sigma,\CC)\not\Vd A$.
%Moreover by Lemma~\ref{lemma:model lemma}, $\Mc$ is a \CNM\ for $\WL$.
%%Therefore it is not the case that $A$ is valid in all $\Wlogic$-models.
%Therefore $A$ is not valid in all $\Wlogic$-models.
there exists a $\WL$-segment $(\Sigma,\CC)$. %, that by
By Def.~\ref{def:can model}, $(\Sigma,\CC)$ belongs to the canonical model $\Mc$ for $\WL$.
Then by Lemma~\ref{lemma:truth lemma}, $(\Sigma,\CC)\not\Vd A$,
and by Lemma~\ref{lemma:model lemma}, $\Mc$ is a $\WL$-model.
%Therefore it is not the case that $A$ is valid in all $\Wlogic$-models.
Therefore $A$ is not valid in all $\Wlogic$-models.
\end{proof}

\section{Discussion and future work}

In this paper, we have defined a family of 14 \const\ modal logics
both proof-theoretically and semantically motivated,
corresponding each to a different classical modal logic.
%From the point of view of the proof-theory, 
On the one hand, the logics correspond to the single-succedent restriction
of standard sequent calculi for classical modal logics.
%, corresponding each to a different classical modal logic.
On the other hand, the same logics are obtained 
%by generalising the forcing conditions of modal formulas to the $\less$-successors
%in order to build in hereditariness into them.
by considering over intuitionistic Kripke models a %simple
natural  generalisation of 
the classical %forcing conditions 
satisfaction clauses for modal formulas in the neighbourhood semantics.
%%through 
%%by means of the $\less$-generalisation of the forcing conditions of modal formulas.
%The main result of this paper is that for the considered logics, the two methodologies %strategies
%return the same systems.
The main result of this paper is that,
despite being mutually independent,
%from each other,
for the considered logics
the two approaches % methodologies %strategies
return exactly the same systems.

%old paragraph
%The choice of considering neighbourhood models is motivated by the possibility to uniformly cover all considered logics,
%which include both normal and non-normal systems.
%However, $\WK$, $\WKD$ and $\WKT$ have
%an equivalent characterisation in terms of \const\ bi-relational models \cite{wij},
%% the same holds
%and
%an analogous characterisation can be given %also 
%for $\WMC$ and its extensions in terms of % by considering 
%relational models with non-normal worlds 
%%(cf.~e.g.~\cite{Priest:2008,dal:JLC}).
%(cf.~e.g.~\cite{Priest:2008}).

In addition, we have provided some preliminary %initial 
analysis of W-logics.
First, we have shown how W-logics
are related to
% relates to %[WITH] 
the corresponding classical modal logics from the point of view of the axiomatic systems:
%%in particular
%each classical logic $\logic$ is obtainable from the corresponding W-logic $\WLL$ by % adding
%extending $\WLL$ with both $A \lor \neg A$ and $\Box A \lor \diam\neg A$.
each classical modal logic %$\logic$ 
considered in this paper can be obtained by extending the corresponding W-logic
 with both excluded middle $A \lor \neg A$ and disjunctive duality $\Box A \lor \diam\neg A$.
Moreover, basing on their sequent calculi we have proved some fundamental properties of W-logics, such as the disjunction property, decidability and Craig interpolation.

%In \cite{simps}, Simpson
Simpson~\cite[Ch.~3]{Simpson:1994} listed %(the following) 
%X conditions that a logic must satisfy in order to be considered an intuitionistc version of a classical modal logic 
some 
%a few 
%conditions 
requirements that one 
%would
%might 
%expect
expects to be satisfied by 
%\intu\ modal logics: 
any \intu\ modal logic:
%some features that one would expect of any \intu\ modal logic, 
%among which
%we mention the following:
%it 
they must be conservative over $\IPL$;
%it 
they must contain all axioms of $\IPL$ (over the whole language) and %it is
be closed under modus ponens;
%it 
they must satisfy the disjunction property;
% (if $A\lor B$ is a theorem, then also one between $A$ and $B$ is);
the modalities must be independent;
the addition of the axiom $A\lor\neg A$ must yield a standard classical modal logic.
%Basing on the sequent calculi, 
Basing on the results presented in this paper,
it is easy to %see
verify that 
%the first four conditions are satisfied by all \tiz{W-logics}, 
all W-logics satisfy the first four requirements, %conditions, 
%By constrast, 
by contrast %\tiz{W-logics}
they do not satisfy the last one.% %condition. %latter one.
\footnote{This requirement has been sometimes criticised as being too strong,
see %\eg~
\cite{Litak} for an argument against this requirement 
based on negative translations.}
However, 
%since we are dealing with modal logics, 
%%why there could not be some modal principle 
%%additional to excluded middle
%%that distinguishes between intuitionistic and classical modalities?
it comes natural to ask whether there could be some modal principle,
additional to excluded middle,
that distinguishes between \const\ % intuitionistic 
and classical modalities.
As a matter of fact, 
it is easy to identify such a principle for W-logics:
%which
%it
as we observed above this principle 
is precisely
$\Box A \lor \diam\neg A$.
%%(for an argument against Simpson's last requirement see \cite{Litak}).} %\nb{ADD REF LITAK}
%(see \cite{Litak} for an argument against Simpson's last requirement based on negative translations).}
%It is easy to see that for W-logics such a principle exists,
%since 
%the classical modal logics %$\logic$ 
%considered in this paper can be obtained by extending the corresponding W-logic %$\WL$
%%with the excluded middle and the principle of 
%not only with excluded middle, but also with
%%the principle of
%disjunctive duality $\Box A \lor \diam\neg A$.

%We point out that this is not an obvious relation between \const\ and classical modal logics.
%For instance, the same % relation 
%does not hold for $\CK$, which
%must be
%extended also with $\neg(\Box A \land \diam \neg A)$ (or equivalently with $\neg\diam\bot$) 
%in order to obtain classical $\K$.

%We point out that this
This  relation between classical and W-logics is not entirely trivial.
For instance, the same relation 
does not hold between $\CK$ and $\K$, in particular $\CK$
must be
extended also with $\neg(\Box A \land \diam \neg A)$ (or equivalently with $\neg\diam\bot$) 
in order to obtain classical $\K$.
Moreover,
%failure of  $\Box A \lor \diam\neg A$ seems
we believe that failure of  $\Box A \lor \diam\neg A$ %can be justified 
is justifiable
%in a constructive context,
from a \const\ perspective,
as it can be seen as a modalised form of excluded middle.
%We also remark
%%notice 
%that Simpson's requirement
%%(coinciding 
%(which coincides with Requirement 3 in Simpson's list \cite{Simpson:1994})
%%has not received 
%%%gained 
%%global consensus
%%%has been sometimes criticised
%%%as being too restrictive
%%as it is 
%has been
%sometimes
%considered too strong
%(see \eg~\cite{Litak} for an argument against Simpson's requirement 
%based on negative translations).

Concerning the semantics of W-logics, the choice of considering neighbourhood models is motivated by the possibility to uniformly cover all considered logics,
which include both normal and non-normal systems.
However, $\WK$, $\WKD$ and $\WKT$ have
an equivalent characterisation in terms of \const\ bi-relational models \cite{wij},
%moreover
and we conjecture that
an analogous characterisation can be given %also 
for $\WMC$ and its extensions in terms of % by considering 
relational models with non-normal worlds 
%(cf.~e.g.~\cite{Priest:2008,dal:JLC}).
(cf.~e.g.~\cite{Priest:2008}).
As a byproduct of this work,
we have provided a new semantics for $\WK$ %which is 
alternative to 
its original relational semantics \cite{wij}
and to the neighbourhood semantics in \cite{Kojima:2012,dal:JPL}.

The possibility to %provide 
define constructive counterparts of both normal and non-normal classical 
logics 
can be seen as providing
%provides 
additional justification for
%of 
the present approach.
%For instance, the justification of \intu\ modal logics ultimately relies on the standard translation of modal formulas
%and validity in \intu first-order logic.
%%To make a comparison, it is not obvous how to extend the definition
%%of \intu\ modal logics to non-normal systems, 
%%as their definition
%%as they
%%ultimately 
%%relies
%%rely
%%on the standard translation of modal formulas into first-order sentences. 
To make a comparison, it is not obvious how to extend the family
of \intu\ modal logics %\cite{x,y}
($\IK$ and extensions) to non-normal systems, 
given that their definition
ultimately  relies
on the standard translation of modal formulas into first-order sentences, % \cite{x,y}. 
which in turn is based on the relational semantics.
%It is not obvious how to extend this approach to non-normal logics.
Interestingly, 
%the resulting \const\ counterparts of non-normal logics are not entirely new.
the \const\ counterparts of non-normal logics that we have obtained are not entirely new.
In particular, %our logics 
$\WM$ and $\WMN$ coincide with the logics %systems 
$\logicnamestyle{IM}$ and $\logicnamestyle{IMN_\Box}$ introduced in \cite{dal:JPL},
where they are given an alternative %\neigh\ 
semantics
with distinct \neigh\ functions for $\Box$ and $\diam$.
%whence as a side result we have provided an alternative semantics for these two systems.
By contrast, $\WMC$ is not equivalent to $\logicnamestyle{IMC}$ in
 \cite{dal:JPL},
since % the first system
$\WMC$ contains \axKdiam\ which is not a theorem of $\logicnamestyle{IMC}$. % the second.
%The logics $\WM$ and $\WMN$ coincide with the systems $\logicnamestyle{IM}$ and $\logicnamestyle{IMN_\Box}$
%introduced in \cite{dal:JPL}, where the two systems are given an alternative neighbourhood semantics
%with two distinct %and suitably related 
%neighbourhood functions for $\Box$ and $\diam$. % which handle the modalities separately.
%By contrast the system $\logicnamestyle{IMC}$ in \cite{dal:JPL} is weaker than $\WMC$ as it does not contain \axKdiam.

%\bigskip

The results presented in this paper can be extended in several directions.
%First, 
%%although they allow %to prove
%%for a proof of decidability,
%G3-style calculi are not adequate to establish good complexity bounds
%for \const\ logics.
In future work we plan to study the complexity of W-logics, possibly extending some optimal calculi for $\IPL$
(G3-style calculi are not adequate to establish good complexity bounds
for \const\ logics).
Moreover, we would like to study whether Iemhoff's proof-theoretical method for proving 
uniform interpolation \cite{Iemhoff:2019} can be adapted to W-logics.
We would also like to define  calculi for W-logics 
%constructing countermodels for non-valid formulas, 
that allow for a direct %construction
extraction of countermodels from failed proofs,
along the lines of \cite{iemhoff2020g4i,dal:tableaux,dal:JLC}.

%
%Finally, 
%proof-search in $\seqWstar$ is not strictly terminating
%because of the rule \ilimp\ 
%which copies the principal implication into the left premiss
%(and similarly \ruleiTbox).
%In future work we would like to define terminating calculi for W-logics,
%possibly also constructing countermodels for non-valid formulas, 
%along the lines of \cite{dal:tableaux,dal:JLC}.

%The results presented in this paper can be extended in several directions.
%First, 
Furthermore,
one can extend the
present analysis to further classical modal logics in order to enrich the family of W-logics,
but also to inspect %TEST
%investigate 
the limits of %the
our approach.
An obvious 
%A clear
limit concerns the logics for which no standard cut-free Gentzen calculi %are available, 
exist,
such as $\Sfive$.
For these logics %it could be worth studying 
one can study whether 
%the same or a similar approach could be applied to
a similar analysis could be based on
 alternative kinds of calculi, 
%such as
like hyper- or nested sequent calculi.
At the same time, it is known that 
%builing-in hereditariness
incorporating hereditariness into the satisfaction clauses % forcing clauses
is not sufficient to provide a semantics for some 
constructive systems, 
%such as those 
this is the case for instance of the logics
with axiom \axfour~\cite{Alechina}.
%For instance, it is known that 
%in order to validate the transitivity axiom 4,
%the $\less$-generalisation of modal clauses in transitive relational models is not sufficient,
%rather one needs a peculiar interaction between $\less$ and $\R$ \cite{x,y}.
%\tiz{Concerning instead 
%%weaker, non-monotonic logics~\cite{chellas...}, 
%logics weaker than $\EM$ (\ie, non-monotonic logics~\cite{chellas...}),
%it seems that the approach returns logics with a very weak form of duality analogous to the one of %the logic
%$\logicnamestyle{IE_1}$ in \cite{dal:JPL},
%but this requires further analysis.
%FINIRE}
Concerning instead 
weaker systems,
it seems that for non-normal logic $\E$ \cite{Chellas:1980}
this approach returns a very weak form of duality analogous to the one of %the logic
$\logicnamestyle{IE_1}$ in \cite{dal:JPL},
but this requires further study.

\vspace{0.4cm}

\noindent \textbf{Acknowledgements.}
This work was supported by the SPGAS and CompRAS projects at the Free University of Bozen-Bolzano, and by the EU H2020 project INODE (grant agreement No 863410).

%% Bibliography
%% Make sure to use the bibliographystyle aiml22.
\bibliographystyle{aiml22}
%\bibliography{aiml22}

\end{document}